\documentclass[a4paper,11pt, reqno]{amsart}
\usepackage[utf8]{inputenc}
\usepackage{amsmath}
\usepackage{amssymb}
\usepackage{amsthm}
\usepackage{a4wide}
\usepackage{paralist}
\usepackage{marvosym}
\usepackage{esint}
\usepackage{dsfont}
\usepackage{mathrsfs} % Added by TF
\usepackage{hyperref} % Added by TF
\usepackage{color} % Added by TF

\newtheorem{theorem}{Theorem}[section]
\newtheorem{lemma}[theorem]{Lemma}%[section] %geaenderte Zaehlung
\newtheorem{proposition}[theorem]{Proposition}%[section] %geaenderte Zaehlung
\newtheorem{corollary}[theorem]{Corollary}%[section] %geaenderte Zaehlung
\theoremstyle{definition}
\newtheorem{definition}[theorem]{Definition}%[section] %geaenderte Zaehlung
\newtheorem*{acknowledgements}{Acknowledgements}

\theoremstyle{remark}
\newtheorem{remark}[theorem]{Remark}%[section] %geaenderte Zaehlung

\numberwithin{equation}{section} % equations are numbered within sections

\def\ba{\begin{array}}
\def\ea{\end{array}}
\def\bea{\begin{eqnarray} \label}
\def\eea{\end{eqnarray}}
\def\be{\begin{equation} \label}
\def\ee{\end{equation}}
\def\bit{\begin{itemize}}
\def\eit{\end{itemize}}
\def\ben{\begin{enumerate}}
\def\een{\end{enumerate}}

\def\lan{\langle}
\def\ran{\rangle}

\def\EE{\mathbb{E}}

\def\NN{\mathbb{N}}
\def\PP{\mathbb{P}}

\def\RR{\mathbb{R}}

\def\a{\alpha}

\def\g{\gamma}

\def\e{\varepsilon}

\def\n{\eta}
\def\hn{\hat\eta} % Added by TF

\def\l{\lambda}

\def\s{\sigma}

\def\ph{\varphi}
\def\o{\omega}
\def\G{\Gamma}

\def\O{\Omega}

\def\bG{\mathbf{G}}
\def\bP{\mathbf{P}}

\def\bY{\mathbf{Y}}
\def\bX{\mathbf{X}}
\def\one{\mathbf{1}} % added by TF

\def\cF{\mathcal{F}}

\def\cL{\mathcal{L}}
\def\cN{\mathcal{N}}

\def\cZ{\mathcal{Z}}

\def\sZ{\mathscr{Z}}

\def\dint{\textup{d}}
\def\tld{\widetilde} % New definition by TF
\title[A four moments theorem for Gamma limits on a Poisson chaos]{A four moments theorem for Gamma limits\\ on a Poisson chaos}

\thanks{The first author has been supported by the Swiss National Science Foundation (SNF) via grant 152609. The second author has been supported by the German Research Foundation (DFG) via SFB-TR 12.}

\author[T. Fissler]{Tobias Fissler}
\address{Tobias Fissler: Institute of Mathematical Statistics and Actuarial Science, University of Bern, Switzerland}
\email{tobias.fissler@stat.unibe.ch}

\author[C. Th\"ale]{Christoph Th\"ale}
\address{Christoph Th\"ale: Faculty of Mathematics, Ruhr University Bochum, Germany}
\email{christoph.thaele@rub.de}

\date{}

\begin{document}

\begin{abstract}
This paper deals with sequences of random variables belonging to a fixed chaos of order $q$ generated by a Poisson random measure on a Polish space. The problem is investigated whether convergence of the third and fourth moment of such a suitably normalized  sequence to the third and fourth moment of a centred Gamma law implies convergence in distribution of the involved random variables. A positive answer is obtained for $q=2$ and $q=4$. The proof of this four moments theorem is based on a number of new estimates for contraction norms. Applications concern homogeneous sums and $U$-statistics on the Poisson space.
\end{abstract}

\maketitle

\section{Introduction}

Probabilistic limit theorems for sequences of multiple stochastic integrals have found considerable attention during the last decade. One of the most remarkable results in this direction is the fourth moment theorem of Nualart and Peccati obtained in the seminal paper \cite{NualartPeccati2005}. It asserts that a sequence of suitably normalized multiple stochastic integrals of order $q\geq 1$ with respect to a Gaussian random measure on a Polish space satisfies a central limit theorem if and only if the sequence of their fourth moments converges to $3$, the fourth moment of a standard Gaussian distribution. This drastic simplification of the method of moments has stimulated a large number of applications, for example to Gaussian random processes or fields, mathematical statistics, random matrices or random polynomials (we refer the reader to the monograph \cite{Nourdin2012Book} and also to the constantly updated web-page \texttt{https://sites.google.com/site/malliavinstein/home} for further details and references).

Besides the fourth moment theorem mentioned above, there is also a `non-central' version dealing with the approximation of a sequence of multiple stochastic integrals by a Gamma-distributed random variable, cf.\ \cite{NourdinPeccati2009a}. Again, the result is a drastic simplification of the method of moments as it delivers convergence in distribution if and only if a certain linear combination of the third and the fourth moment of the involved random variables converges to the corresponding expression for Gamma random variables. In view of normalization conditions we see that in fact the first four moments of the random variables are involved, which gives rise to the name `four moments theorem' for such a result. To simplify the terminology, we will also speak about a four moments theorem in the case of normal approximation.

The present paper asks whether a similar non-central limit theorem is available for sequences of multiple stochastic integrals with respect to a Poisson random measure on a Polish space. In this set-up, a central four moments theorem has been derived by Lachi\`eze-Rey and Peccati in \cite{Lachieze-ReyPeccati2013a} under an additional sign condition (see also \cite{EichelsbacherThaele2013}), which, on the Poisson space, seems to be unavoidable. While Gamma approximation on the Poisson space in the spirit of the Malliavin-Stein method has been dealt with in \cite{PeccatiThaeleGamma}, the problem of a four moments theorem similar to that for Gaussian multiple stochastic integrals mentioned above remained open in general. The main result of our paper, Theorem \ref{thm:main_theorem}, delivers a four moments theorem for sequences of Poisson stochastic integrals of order $q=2$ and $q=4$. For this reason, the present work can be seen as a natural continuation of \cite{PeccatiThaeleGamma}, where the case $q=2$ has already been settled under additional assumptions, which we are able to overcome. The proof of our four moments theorem relies on a couple of new estimates for norms of so-called contraction kernels and the combinatorially involved multiplication formula for stochastic interals on the Poisson space. It is precisely this combinatorial complexity which forces that our proof yields a positive result only for sequences of Poisson stochastic integrals of order $q=2$ and $q=4$. However, all intermediate steps in our proof will be formulated for general $q\geq 2$ to make as transparent as possible and to highlight, in which argument the restrictive condition on the order of the integrals arises. The main difference between the central and the non-central version of the four moments theorem is that in the non-central case one has to deal with a linear combination of the third and the fourth moment of the stochastic integrals, while the central case only requires an analysis of the fourth moment. Even under additional conditions on the integrands, this leads to difficulties, which we can overcome only for $q=2$ and $q=4$. We have to leave it as an open problem for future research to extend our result to arbitrary $q$ by other methods.

The main result of our paper is applied to a universality question for homogeneous sums on a Poisson chaos as well as to a non-central analogue of de Jong's theorem for completely degenerate $U$-statistics of order two and four. This partially complements the results for Gamma and normal approximation obtained in \cite{EichelsbacherThaele2013,PeccatiThaeleGamma} and \cite{PeccatiZheng2013}. We emphasize in this context that limit theorems for non-linear functionals of Poisson random measures have recently found numerous applications especially in geometric probability or stochastic geometry \cite{EichelsbacherThaele2013,Lachieze-ReyPeccati2013a,LPS,LPST,PeccatiThaeleGamma,STScaling,STFlats} and in the theory of L\'evy processes \cite{EichelsbacherThaele2013,LPS,PeccatiSoleTaqquUtzet2010,PeccatiZhengMulti}.

\medspace

Our paper is structured as follows. In Section \ref{sec:Preliminaries}, we introduce and collect necessary background material. To contrast our results with those available for Gaussian multiple stochastic integrals, we shall present them in the context of completely random measures, which captures both settings. Our main results are the content of Section \ref{sec:4Moments}, while Section \ref{sec:applications} contains applications to homogeneous sums and $U$-statistics. The proof of Theorem \ref{thm:main_theorem} is presented in the final Section \ref{sec:ProofMainThm}.

\section{Preliminaries}\label{sec:Preliminaries}

In this section, we introduce the basic definitions, mainly related to Poisson stochastic integrals. For further details and background material we refer the reader to the monograph \cite{PeccatiTaqquBook} as well as to the papers \cite{NualartVives1990, PeccatiSoleTaqquUtzet2010}.

\subsection{Completely random measures}

Without loss of generality, we assume that all objects are defined on a common probability space $(\O, \cF, \PP)$. Let $\cZ$ denote a Polish space with Borel $\s$-field $\sZ$, which is equipped with a non-atomic $\s$-finite measure $\mu$. We define the class $\sZ_\mu = \{ B\in \sZ \colon \mu(B) <\infty\}$ and let $\ph = \{\ph(B) \colon B\in \sZ_\mu\}$ indicate a \emph{completely random measure} on $(\cZ, \sZ)$ with control measure $\mu$. That is, $\ph$ is a set of random variables such that
\begin{enumerate}[(i)]
\item
for every collection of pairwise disjoint elements $B_1, \ldots, B_n \in \sZ_\mu$, the random variables $\ph(B_1), \ldots, \ph(B_n)$ are independent;
\item
for every $B, C\in \sZ_\mu$, one has the identity $\EE[\ph(B)\ph(C)] = \mu(B\cap C)$.
\end{enumerate}
If $\EE[\ph(B)] = 0$ and $\ph(B)\in L^2(\PP)$ (i.e., $\ph(B)$ is square-integrable with respect to $\PP$) for every $B\in \sZ_\mu$, then the mapping $\sZ_\mu \to L^2(\PP), B\mapsto \ph(B)$, is $\s$-additive in the sense that for every sequence $(B_n)_{n\geq 1}$ of pairwise disjoint elements of $\sZ_\mu$, one has that 
\be{eqn:sigma-additivity}
\ph\left(\bigcup_{n=1}^\infty B_n\right) = \sum_{n=1}^\infty \ph(B_n) \qquad \PP\text{-a.s.},
\ee
where the right-hand side converges in $L^2(\PP)$.
By $\s(\ph)$ we denote the $\s$-field generated by $\ph$.

In this paper, we shall deal with two special and prominent instances of completely random measures, namely a centred Gaussian and a compensated Poisson measure.
\begin{enumerate}
\item
A \emph{centred Gaussian measure} with control measure $\mu$ is denoted by $G$ and is a completely random measure such that the elements of $G$ are jointly Gaussian and centred.
\item
A \emph{compensated Poisson measure} with control measure $\mu$ is indicated by $\hn$ and is a completely random measure such that for every $B\in \sZ_\mu$, $\hn(B) \stackrel{d}{=} \n(B) - \mu(B)$, where $\n(B)$ is a Poisson random variable with mean $\mu(B)$.
\end{enumerate}
By definition, both $G$ and $\hn$ are centred families in $L^2(\PP)$, implying that \eqref{eqn:sigma-additivity} is satisfied. Moreover, for $\PP$-almost every $\o\in\O$, $\hn(\cdot, \o)$ is a signed measure on $(\cZ, \sZ)$, while $G$ does not satisfy this property, cf.\ \cite[Example 5.1.7 (iii)]{PeccatiTaqquBook}.

\subsection{$L^2$-spaces}

Let $q\ge1$ be an integer. We shall use the shorthand notation $L^2(\mu^q)$ for the space $L^2(\cZ^q, \sZ^{q}, \mu^q)$ of (deterministic) functions that are square-integrable with respect to $\mu^q$. $L_s^2(\mu^q)$ stands for the subspace of $L^2(\mu^q)$ consisting of symmetric functions, i.e., functions that are $\mu^q$-a.e.\ invariant under permutations of their arguments. For $f,g\in L^2(\mu^q)$ we define the scalar product $ \lan f ,g \ran_{L^2(\mu^q)} = \int_{\cZ^q} fg \,\dint \mu^q$ and the norm $\|f\|_{L^2(\mu^q)} = \lan f ,f \ran_{L^2(\mu^q)}^{1/2}$. If there is no risk of confusion, we suppress in what follows the dependency on $q$ and $\mu$, and merely write $\lan\, \cdot\,, \,\cdot \,\ran$ and $\| \cdot\|$, respectively. Moreover, let $L^2(\s(\ph), \PP)$ denote the space of all square-integrable functionals of $\ph$, where $\ph$ is either a Poisson measure $\hn$ or a Gaussian measure $G$. If $F\in L^2(\s(\ph), \PP)$, we shall sometimes write $F = F(\ph)$ in order to underpin the dependency of $F$ on $\ph$. As a convention, we shall use lower case variables for elements of $L^2(\mu^q)$ and capitals for elements of $L^2(\s(\ph), \PP)$. Finally, we introduce the space $L^2(\PP, L^2(\mu)) = L^2(\O\times \cZ, \cF\otimes \sZ, \PP\otimes \mu)$ as the space of all jointly square-integrable measurable mappings $u\colon \O\times \cZ \to \RR$. If $u,v\in L^2(\PP, L^2(\mu))$, their scalar product is defined as $\lan u, v\ran_{L^2(\PP, L^2(\mu))} = \int_\Omega\int_\cZ u(\omega,z)v(\omega,z)\mu(\dint z)\PP(\dint\omega)$ and we denote by $\|\,\cdot\,\|_{L^2(\PP, L^2(\mu))}$ the norm induced by it.

\subsection{Multiple stochastic integrals}

Let $\ph=\hn$ or $\ph= G$.
For every integer $q\ge1$ we denote the \emph{multiple stochastic integral} of order $q$ with respect to $\ph$ by $I_q^\ph$. It is a mapping $I_q^\ph \colon L_s^2(\mu^q) \to L^2(\s(\ph), \PP)$, which is linear and continuous. Additionally, for $f\in L_s^2(\mu^q)$, the random variable $I_q^\ph(f)$ is centred. Moreover, the multiple stochastic integral satisfies the \emph{It\^o isometry}
\be{eqn:isometry}
\EE[I_p^\ph(f) I_q^\ph(g)] =
\begin{cases}
0 & \text{if } q\neq p \\
q!\lan f,g\ran_{L_s^2(\mu^q)} & \text{if } q=p
\end{cases}
\ee
for any integers $p,q\ge1$ and $f\in L_s^2(\mu^p)$, $g\in L_s^2(\mu^q)$. For general $f\in L^2(\mu^q)$, we put $I_q^\ph(f) = I_q^\ph(\tld f)$, where 
\[
\tld f(z_1, \ldots, z_q) = \frac{1}{q!} \sum_{\pi\in\Pi_q} f(z_{\pi(1)}, \ldots, z_{\pi(q)})
\]
is the canonical symmetrization of $f$, and $\Pi_q$ is the group of all $q!$ permutations $\pi$ of the set $\{1,\ldots, q\}$. We emphasize that due to Jensen's inequality and the convexity of norms, we have the inequality $\|\tld f\| \le \| f\|$. As a convention, we set $I_0^\ph \colon \RR \to \RR$ equal to the identity map on $\RR$.

Since this article is mostly concerned with Poisson integrals, we shall write $I_q$ instead of $I_q^{\hn}$.

\subsection{Chaos decomposition}

The It\^o isometry in \eqref{eqn:isometry} formalizes an orthogonality relation between multiple stochastic integrals of different order. This induces the following so-called \emph{chaos decomposition} (see \cite{NualartVives1990}):
\be{eqn:decomposition_a}
L^2(\s(\ph), \PP) = \bigoplus_{q=0}^\infty W_q^\ph\,,
\ee
where $W_0^\ph= \RR$ and $W_q^\ph= \{I_q^\ph(f)\colon f\in L_s^2(\mu^q) \}$ for $\ph=\hn$ or $\ph = G$, $q\ge1$. Depending on the choice of $\ph$, we shall often use the terms \emph{Poisson chaos} and \emph{Gaussian chaos of order $q$} for $W_q^\ph$, respectively.

A consequence of \eqref{eqn:decomposition_a} is that any $F\in L^2(\s(\ph), \PP)$, with $\ph = \hn$ or $\ph=G$, admits a chaotic decomposition 
\[
%\be{eqn:decomposition}
F= \EE[F] + \sum_{q=1}^\infty I_q^\ph(f_q)\,,
%\ee
\]
where the kernels $f_q \in L_s^2(\mu^q)$ are unique $\mu^q$-a.e.\ and the series converges in $L^2(\PP)$.

\subsection{Contractions}

Fix integers $p, q \ge1$ and functions $f\in L_s^2(\mu^p)$, $g\in L_s^2(\mu^q)$. For any $r\in\{0,\ldots, p\wedge q\}$ and $\ell \in \{1,\ldots, r\}$ we define the \emph{contraction} $f\star_r^\ell g\colon 
\cZ^{p+q -r - \ell} \to \RR$ which acts on the tensor product $f\otimes g$ and reduces the number of variables from $p+q$ to $p+q - r - \ell$ in the following way: $r$ variables are identified and among these, $\ell$ are integrated out with respect to $\mu$. More formally,
\[
\begin{split}
f\star_r^\ell g& (\g_1, \ldots, \g_{r-\ell}, t_1, \ldots, t_{p-r}, s_1, \ldots, s_{q-r}) \\
&= \int_{\cZ^\ell} f(z_1, \ldots, z_\ell, \g_1, \ldots, \g_{r-\ell}, t_1, \ldots, t_{p-r}) \\
&\qquad \times g(z_1, \ldots, z_\ell, \g_1, \ldots, \g_{r-\ell},s_1, \ldots, s_{q-r})\,
\mu^\ell (\dint( z_1 \ldots  z_\ell))\,,
\end{split}
\]
and for $\ell=0$ we put
\[
\begin{split}
f\star_r^0 g& (\g_1, \ldots, \g_{r}, t_1, \ldots, t_{p-r}, s_1, \ldots, s_{q-r}) \\
&= f(\g_1, \ldots, \g_{r}, t_1, \ldots, t_{p-r}) g( \g_1, \ldots, \g_{r},s_1, \ldots, s_{q-r})\,.
\end{split}
\]
Note that even if $f$ and $g$ are symmetric, the contraction $f\star_r^\ell g$ is not necessarily symmetric. We denote the canonical symmetrization by
\[
f \,\tld \star_r^\ell g (z_1, \ldots, z_{p+q - r - \ell}) = \frac{1}{(p+q-r-\ell)!} \sum_{\pi\in\Pi_{p+q-r-\ell}} f\star_r^\ell g(z_{\pi(1)},\ldots, z_{\pi(p+q - r - \ell)}).
\]
%We illustrate this definition with the following example where $f,g\in L^2(\mu^2)$:
%\be{eqn:contractions q=2}
%\left.
%\begin{split}
%f\star_0^0 g (t_1,t_2, s_1, s_2) &= f(t_1, t_2) g(s_1,s_2), \\
%f\star_1^0 g (\g,t,s) &= f(\g,t)g(\g,s) , \\
%f\star_2^0 g (\g_1, \g_2) &= f(\g_1,\g_2) g(\g_1,\g_2), \\
%f\star_1^1 g (t,s) &= \int_\cZ f(z,t) g(z,s) \mu(\dint z), \\
%f\star_2^1 g (\g) &= \int_\cZ f(z,\g) g(z,\g) \mu(\dint z), \\
%f\star_2^2 g  &=\lan f, g\ran_{L^2(\mu^2)} = \int_{\cZ^2} f(z_1,z_2)g(z_1,z_2) \mu^2(\dint z_1 \dint z_2).
%\end{split}
%\right\}
%\ee
We also emphasize that for $f\in L_s^2(\mu^p)$ and $g\in L_s^2(\mu^q)$, the contraction $f\star_r^\ell g$ is neither necessarily well-defined nor necessarily an element of $L^2(\mu^{p+q-r-\ell})$. At least, by using the Cauchy-Schwarz inequality, we can deduce that $f\star_r^r g \in L^2(\mu^{p+q - 2r})$ for any $r\in\{0,\ldots, p\wedge q\}$. For this reason and to circumvent any complications in the calculations, we make the following technical assumptions.

\subsection{Technical assumptions (A)}\label{subsec:Assumptions}

We use the same set of technical assumptions as in \cite{Lachieze-ReyPeccati2013a, PeccatiSoleTaqquUtzet2010, PeccatiThaeleGamma}. For a detailed explanation of the conditions and their consequences, we refer to these works. 

For a sequence $F_n = I_q(f_n)$ of multiple integrals of fixed order $q\ge1$ with $f_n \in L^2(\mu_n^q)$ for every $n\ge1$ (we allow the non-atomic and $\s$-finite measure to vary with $n$), we assume that the following three technical conditions are satisfied:
\begin{enumerate}
\item
for any $r\in \{1, \ldots, q\}$, the contraction $f_n \star_q^{q-r} f_n$ is an element of $L^2(\mu_n^r)$;
\item
for any $r\in \{1, \ldots, q\}$, $\ell \in\{1, \ldots, r\}$ and $(z_1, \ldots, z_{2q-r-\ell}) \in \cZ^{2q-r-\ell}$, we have that $(|f_n|\star_r^\ell |f_n|)(z_1, \ldots, z_{2q-r-\ell})$ is well-defined and finite;
\item
for any $k\in \{0,\ldots, 2(q-1)\}$ and any $r$ and $\ell$ satisfying $ k = 2(q-1) - r - \ell$, we have that 
\[
\int_\cZ \sqrt{\int_{\cZ^k} \big( f_n(z, \cdot) \star_r^\ell f_n(z,\cdot) \big)^2 \dint \mu_n^k} \, \mu_n(\dint z) < \infty\,.
\]
\end{enumerate}

\subsection{Multiplication formula}

A very convenient property of multiple stochastic integrals is that one can express the product of two such integrals as a linear combination of multiple integrals of contraction kernels. More precisely, we have the following multiplication formula for Poisson integrals, which is taken from \cite[Proposition 6.5.1]{PeccatiTaqquBook}.

\begin{lemma}[Multiplication formula for Poisson integrals]\label{lemma:multiplication}
Let $f\in L_s^2(\mu^p)$ and $g\in L_s^2(\mu^q)$, $p,q\ge1$. Suppose that $f\star_r^\ell g \in L^2(\mu^{p+q - r - \ell})$ for every $r \in \{0, \ldots, p\wedge q\}$ and every $\ell\in\{0, \ldots, r\}$. Then
\be{eqn:multiplication}
I_p(f) I_q(g) = \sum_{r=0}^{p\wedge q} r! \binom pr \binom qr \sum_{\ell=0}^r I_{p+q - r - \ell} (f\,\tld\star_r^\ell g).
\ee
\end{lemma}
We remark that if a kernel $f\in L_s^2(\mu^q)$ satisfies the technical assumptions (A), the assumptions of Lemma \ref{lemma:multiplication} are automatically satisfied if $g=f$, implying that $I_q(f)^2 \in L^2(\s(\hn), \PP)$. 
To simplify our notation, for $f\in L_s^2(\mu^q)$ we put $G_0^q f = q!\|f\|^2$ and 
\be{eqn:G_p^q}
G_p^q f = \sum_{r=0}^q \sum_{\ell=0}^r \one(2q - r - \ell=p) r! \binom qr^2 \binom r\ell f \,\tld\star_r^\ell f 
\ee
for $p\in\{1, \ldots, 2q\}$. In other words, the operator $G_p^q$ turns a function of $q$ variables into a function of $p$ variables. We can now re-write \eqref{eqn:multiplication} in a simplified form as
\[
I_q(f)^2 = \sum_{p=0}^{2q} I_p(G_p^q f)
\]
with $I_0(G_0^q f) = G_0^q f =  q!\|f\|^2$.

The multiplication formula paves the way for the computation of moments of multiple stochastic integrals. In particular, we have the following expressions for the third and the fourth moment of a multiple Poisson integral.

\begin{lemma}[Third and fourth moment of Poisson integrals]\label{lemma:Third and fourth moment}
Fix an integer $q\ge1$. Let $f \in L_s^2(\mu^q)$ such that the technical assumptions (A) are satisfied. Then $I_q(f) \in L^4(\PP)$. Moreover, we have that
\begin{align}
\label{eqn:3moment}
\EE[I_q(f)^3] &= q! \sum_{r=0}^q \sum_{\ell=0}^r \one(q=r+\ell) r! \binom qr^2 \binom r\ell  \lan f \,\tld \star_{r}^{\ell} f,f\, \ran, 
\\
\EE[I_q(f)^4] &= \sum_{p=0}^{2q} p! \|G_p^q\, f \|^2.
\label{eqn:4moment}
\end{align}
\end{lemma}
\begin{proof}
The technical assumptions (A) ensure that all symmetrized contraction kernels $f \,\tld\star_r^\ell f $ appearing in \eqref{eqn:3moment} and \eqref{eqn:4moment} are elements of $L^2(\mu^{2q - r - \ell})$, which implies that the third and the fourth moment of $I_q(f)$ are finite. The explicit formulae in \eqref{eqn:3moment} and \eqref{eqn:4moment} follow directly from the isometry property \eqref{eqn:isometry} and the multiplication formula \eqref{eqn:multiplication}.
\end{proof}

\begin{remark}
Note that for even $q\ge2$, \eqref{eqn:3moment} reduces to
\be{eqn:3moment even}
\EE[I_q(f)^3] = q! \sum_{r=q/2}^q  r! \binom qr^2 \binom{r}{q-r}  \lan f \,\tld \star_{r}^{q-r} f,f\, \ran\,.
\ee
\end{remark}

\begin{remark}\label{remark:multiplication Gaussian}
There is also a multiplication formula for the Gaussian case.
It reads 
\[
%\be{eqn:multiplication Gaussian}
I_p^G(f) I_q^G(g) = \sum_{r=0}^{p\wedge q} r! \binom pr \binom qr  I_{p+q - 2r}^G (f\,\tld\star_r^r g),
%\ee
\]
where $p,q\ge1$ and $f\in L^2_s(\mu^p)$, $g\in L^2_s(\mu^q)$, see \cite[Proposition 6.4.1]{PeccatiTaqquBook}. As a consequence, we see that the third and fourth moment of a Gaussian multiple integral have a more compact form compared to the Poisson case. Indeed, for an integer $q\ge1$ and $f\in L^2_s(\mu^q)$, one has that 
\begin{align}
\label{eqn:3moment Gaussian}
\EE[I_q^G(f)^3] &= \frac{(q!)^3}{(q/2!)^2}  \lan f \,\tld \star_{q/2}^{q/2} f,f\, \ran\, \one(q \text{ is even})\,,
\\
%\label{eqn:4moment Gaussian}
\nonumber
\EE[I_q^G(f)^4] &= \sum_{r=0}^q (r!)^2 \binom{q}{r}^4 (2q -2r)! \| f\,\tld\star_r^r f\|^2\,.
\end{align}
In particular, the third moment of a Gaussian integral of odd order vanishes, while this is in general not the case for a Poisson integral.
\end{remark}

\section{Four moments theorems}\label{sec:4Moments}

This section contains the main results of our paper, namely a four moments theorem for Gamma approximation on a Poisson chaos of fixed order. To allow for an easier comparison with the existing literature, we first recall known results on a Gaussian chaos and also a version of the four moments theorem for normal approximation on a Poisson chaos.

\subsection{Four moments theorems on a Gaussian chaos}

The classical method of moments yields a central limit theorem for a normalized sequence of random variables under the condition that all moments converge to those of the standard Gaussian distribution. The four moments theorem on a Gaussian chaos is a drastical simplification of the method of moments as it provides a central limit theorem for a sequence of normalized Gaussian multiple stochastic integrals under the much weaker condition that only the fourth moment converges to $3$ (which is the fourth moment of the standard Gaussian distribution). Alternatively, this statement can be re-formulated in terms of the convergence of norms of contractions. % and we recall the following version for easier reference and to simplify comparison with the Poisson case treated below.
In what follows we write $X\sim\cL$ if a random variable $X$ has distribution $\cL$.

\begin{theorem}[see Theorem 1 in \cite{NualartPeccati2005}]\label{thm:GaussianNormal}
Fix an integer $q\ge2$ and let $\{f_n\colon n\ge1\}\subset L_s^2(\mu^q)$ be such that 
\[
\lim_{n\to\infty} q! \|f_n\|^2 = \lim_{n\to\infty}\EE[I_q^G(f_n)^2] = 1\,.
\] 
Further, let $N\sim\cN(0,1)$ be a standard Gaussian random variable.
Then the following three assertions are equivalent:
\begin{enumerate}[(i)]
\item
As $n\to\infty$, the sequence $\{I_q^G(f_n)\colon n\ge1\}$ converges in distribution to $N$;
\item
$\lim\limits_{n\to\infty} \EE[I_q^G(f_n)^4] =3$;
\item
$\lim\limits_{n\to\infty} \|f_n \star_r^r f_n \|=0$ for every $r\in\{1, \ldots, q-1\}$.
\end{enumerate}
\end{theorem}

In the subsequent work \cite{NourdinPeccati2009a}, the authors have shown a `non-central' version of Theorem \ref{thm:GaussianNormal} where the limiting distribution is a centred Gamma distribution. To state the result properly, let us recall the formal definition of the latter limit law.

\begin{definition}[Centred Gamma distribution]\label{def:Centred Gamma distribution}
A random variable $Y$ has a centred Gamma distribution $\overline\G_\nu$ with parameter $\nu>0$, if 
\[
Y \stackrel{d}{=} 2X - \nu,
\]
where $X$ has the usual Gamma law with mean and variance both equal to $\nu/2$ and where $\stackrel{d}{=}$ stands for equality in distribution. The probability density of $\overline\G_\nu$ is given by
\[
g_\nu(x) = \frac{2^{-\nu/2}}{\G(\nu/2)}(x+\nu)^{\nu/2 - 1} e^{-(x+\nu)/2} \one(x>-\nu),
\]
and the the first four moments of $Y$ are
\begin{align*}
\EE[Y] &= 0\,,\quad \EE[Y^2] = 2\nu\,,\quad \EE[Y^3] = 8\nu\,, \quad \EE[Y^4] = 12\nu^2 + 48\nu\,.
\end{align*}
\end{definition}

%\item
%If $\nu$ is an integer, $\overline \G_\nu$ coincides with a centred $\chi^2$-distribution with $\nu$ degrees of freedom. That is, if $Y\sim\overline \G_\nu$, then
%\[
%Y \stackrel{d}{=} \sum_{i=1}^\nu \big(N_i^2 - 1\big)\,,
%\]
%where $(N_1, \ldots, N_\nu)$ is a $\nu$-dimensional vector of i.i.d.\ standard Gaussian random variables.
%\end{enumerate}
%\end{remark}

We are now in the position to re-phrase the following non-central analogue of Theorem \ref{thm:GaussianNormal}.

\begin{theorem}[see Theorem 1.2 in  \cite{NourdinPeccati2009a}]\label{thm:Gaussian-Gamma}
Let $\nu>0$ and fix an even integer $q\ge2$. Let $\{f_n\colon n\ge1\}\subset L_s^2(\mu^q)$ be such that 
\[
\lim_{n\to\infty} q! \|f_n\|^2 = \lim_{n\to\infty} \EE[I_q^G(f_n)^2] = 2\nu\,.
\] 
Further, let $Y\sim\overline\G_\nu$ be a centred Gamma-distributed random variable with parameter $\nu$.
Then the following three assertions are equivalent:
\begin{enumerate}[(i)]
\item
As $n\to\infty$, the sequence $\{I_q^G(f_n)\colon n\ge1\}$ converges in distribution to $Y$;
\item
$\lim\limits_{n\to\infty} \EE[I_q^G(f_n)^4]  - 12\EE[I_q^G(f_n)^3]= 12\nu^2 - 48\nu$;
\item
$\lim\limits_{n\to\infty}\|f_n \star_r^r f_n\| = 0$ for every $r \in \{ 1, \ldots, q-1\}\setminus\{q/2\}$,  and\\
$\lim\limits_{n\to\infty}\|f_n \,\tld\star_{q/2}^{q/2} f_n - c_q \,f_n\| = 0$ with $c_q = \frac{4}{(q/2)! \binom{q}{q/2}^2}$.
\end{enumerate}
\end{theorem}

It is a characterizing feature of the centred Gamma-distribution that the so-called `middle-contraction' $f_n\star_{q/2}^{q/2}f_n$ plays a special role in condition (iii). The fact that the middle-contraction does not vanish goes hand in hand with the appearance of the third moment in condition (ii), recall \eqref{eqn:3moment Gaussian}.

%\begin{remark}
%A sequence $\{I_q(f_n)\colon n\ge1\}$ of Gaussian integrals of odd order $q$ cannot have a centred Gamma distribution as limiting distribution. The reason for this is that $\EE[I_q(f_n)^3]=0$, while $\EE[Y^3]=8\nu>0$ if $Y\sim\overline\G_\nu$.
%To see this, let $q\ge1$ be odd. Due to formula \eqref{eqn:3moment Gaussian}, we have that $\EE[I_q^G(f_n)^3]=0$ for all $f_n\in L_s^2(\mu^q)$. But with respect to a hypercontractivity argument (see \cite[Proposition 2.7.2]{PeccatiTaqquBook}), the sequence $\{I_q^G(f_n)^3\colon n\ge1\}$ is uniformly integrable. Assuming that the law of $\{I_q^G(f_n)\colon n\ge1\}$ converges to $\overline \G_\nu$ with $\nu>0$, implies that $\lim_{n\to\infty} \EE[I_q^G(f_n)^3]=8\nu>0$ which gives a contradiction.
%
%In contrast to Gaussian integrals, it is a priori not clear whether a sequence of Poisson integrals of odd order can converge to a centred Gamma distribution. However, we will not pursue this question here further.
%\end{remark}

\subsection{Four moments theorems on a Poisson chaos}\label{subsec:$4$th-moment theorems in the Poisson chaos}

We now turn to four moments theorems on a Poisson chaos of fixed order $q\geq 2$. To this end, let, for each $n\geq 1$, $\mu_n$ be a $\s$-finite non-atomic measure on $(\cZ,\sZ)$ and denote by $\hat{\eta}_n$ a compensated Poisson random measure with control $\mu_n$. Further let $\{f_n:n\geq 1\}$ be a sequence of symmetric function such that $f_n$ is square-integrable with respect to $\mu_n^q$ for each $n\geq 1$. In this set-up, $\|f_n\|$ denotes the norm of $f_n$ with respect to $\mu_n^q$, and $f_n\star_r^\ell f_n$ stands for the contraction taken with respect to $\mu_n$. Finally, define $F_n=I_q(f_n)$, where for each $n$ the stochastic integral is with repsect to $\hat \eta_n$.

As in the Gaussian case discussed in the previous section, we start with the case of a standard normal limiting distribution.

\begin{theorem}[see Theorem 3.12 in \cite{Lachieze-ReyPeccati2013a}]\label{thm:PoissonNormal}
Let $\{\mu_n:n\geq 1\}$ be a sequence of $\s$-finite and non-atomic measures such that $\lim\limits_{n\to\infty}\mu_n(\cZ)=\infty$ and fix $q\geq 2$. Let $\{f_n\colon n\ge1\} \subset L_s^2(\mu_n^q)$ be a sequence such that for each $n\ge1$ either $f_n\geq 0$ or $f_n\leq 0$. Suppose that the technical assumptions (A) and the normalization condition
\be{eqn:normalization}
\lim_{n\to\infty} q! \|f_n\|^2 = \lim_{n\to\infty}\EE[I_q(f_n)^2] = 1
\ee
are satisfied. Further, suppose that $\{I_q(f_n)^4:n\geq 1\}$ is uniformly integrable and let $N\sim\cN(0,1)$ be a standard Gaussian random variable. Then the following three assertions are equivalent:
\begin{enumerate}[(i)]
\item As $n\to\infty$, the sequence $\{I_q(f_n):n\geq 1\}$  converges in distribution to $N$;
\item $\lim\limits_{n\to\infty}\EE[I_q(f_n)^4]=3$;
\item $\lim\limits_{n\to\infty}\|f_n\star_r^\ell f_n\|=0$ for all $r\in \{1,\ldots, q\}$ and $\ell \in\{1, \ldots, r \wedge (q-1)\}$, and\\ $\lim\limits_{n\to\infty}\|f_n \|_{L^4(\mu^q)}=0$.
\end{enumerate}
\end{theorem}

Let us comment on the differences between Theorem \ref{thm:GaussianNormal} and Theorem \ref{thm:PoissonNormal}. 
\begin{itemize}
\item[(1)] In the Poisson case one has to ensure that the involved control measures are infinite measures, at least in the limit, as $n\to\infty$. The reason for this is that otherwise, the normalization \eqref{eqn:normalization} and the condition that $\lim\limits_{n\to\infty}\|f_n \|_{L^4(\mu^q)}=0$ are mutually exclusive, see also the remark after Assumption N in \cite{PeccatiTaqqu2008} for a brief discussion of this problem.

\item[(2)] One has to assume that the functions $f_n$ have a constant sign, that is for each $n\ge1$ either $f_n\geq 0$ or $f_n\leq 0$. The reason for this is that in the Poisson case, besides of the contraction norms $\|f_n\star_r^\ell f_n\|$, also scalar products of the form $\lan f_n\star_{r_1}^{\ell_1}f_n,f_n\star_{r_2}^{\ell_2}f_n\ran$ enter the expression of the fourth moments $\EE[I_q(f_n)^4]$. The sign condition then allows to control the signs of these scalar products, which rules out cancellation effects.

\item[(3)] In the Poisson case, one also has to assume that the sequence $\{I_q(f_n)^4\colon n\ge 1\}$ is uniformly integrable, while in the Gaussian case, this condition is automatically fulfilled thanks to the hypercontractivity property of Gaussian integrals (see e.g. \cite[Theorem 2.7.2]{Nourdin2012Book}). This is needed to ensure that the convergence in distribution of $I_q(f_n)$ to $N$ implies the convergence of the first four moments.
\end{itemize}

For general $q\geq 2$ and general sequences $\{f_n\colon n\geq 1\}\subset L_s^2(\mu_n^q)$ there is no version of a four moments theorem on a Poisson chaos relaxing one of the conditions discussed above. However, for $q=2$ the sign condition is not necessary as shown by Theorem 2 in \cite{PeccatiTaqqu2008}. Moreover, for general $q\geq 2$ and if the sequence $\{f_n:n\geq 1\}$ is \textit{tamed} (see Definition \ref{def:tamed sequences} below), Theorem 3.2 in \cite{PeccatiZheng2013} provides a four moments theorem without a sign condition. In this case, also condition (iii) can be relaxed by assuming -- besides the condition on the $L^4$-norm of $f_n$ -- only that $\lim\limits_{n\to\infty}\|f_n\star_r^rf_n\|=0$ for all $r\in\{1,\ldots,q-1\}$.

After having discussed the four moments theorem for normal approximation on the Poisson space, we now turn to the main result of the present work, namely a version of Theorem \ref{thm:Gaussian-Gamma} for Poisson integrals of order $q=2$ and $q=4$. The reason for this rather restrictive condition on the order of the involved integrals will be discussed below.

\begin{theorem}[Four moments theorem for Poisson integrals] \label{thm:main_theorem}
Fix $\nu>0$. Let $q\geq 2$ be even and $\{f_n\colon n\ge1\} \subset L_s^2(\mu_n^q)$ be a sequence satisfying the technical assumptions (A) and the normalization condition
\[
\lim_{n\to\infty} q!\|f_n\|^2 = \lim_{n\to\infty}\EE[I_q(f_n)^2] = 2\nu\,.
\]
Furthermore, let the sequence $\{I_q(f_n)^4\colon n\ge1\}$ be uniformly integrable and let $Y\sim\overline\G_\nu$ be a random variable following a centred Gamma distribution with parameter $\nu$. If one of the conditions
\begin{enumerate}
\item
$q=2$ and $\lim\limits_{n\to\infty} \|f_n^2\|=0$,
\item
$q=4$ and $f_n\le 0$ for all $n\ge1$
\end{enumerate}
is satisfied, then the following three assertions are equivalent:
\begin{enumerate}[(i)]
\item
As $n\to\infty$, the sequence $\{I_q(f_n)\colon n\ge1\}$ converges in distribution to $Y$;
\item
$\lim\limits_{n\to\infty} \EE[I_q(f_n)^4]  - 12\EE[I_q(f_n)^3]= 12\nu^2 - 48\nu$;
\item
$\lim\limits_{n\to\infty}\|f_n \star_r^\ell f_n\|=0$ for all $r\in \{1,\ldots, q\}$ and $\ell \in\{1, \ldots, r \wedge (q-1)\}$ such that	$(r,\ell)\neq (q/2,q/2),$
$\lim\limits_{n\to\infty}\|f_n \|_{L^4(\mu_n^q)}=0$, and 
$\lim\limits_{n\to\infty}\|f_n \,\tld\star_{q/2}^{q/2} f_n - c_q \,f_n\| = 0$ with $c_q = \frac{4}{(q/2)! \binom{q}{q/2}^2}$.
\end{enumerate}
\end{theorem}

\begin{remark}
Under condition (a), Theorem \ref{thm:main_theorem} is a version of Proposition 2.9 in \cite{PeccatiThaeleGamma}. However, in that paper one has to assume that for each $n\geq 1$ the reference measure $\mu_n$ is finite. As discussed earlier in this section, this is a quite restrictive assumption. We provide a proof which circumvents this technicality.
\end{remark}

The implication (i) $\Longrightarrow$ (ii) of Theorem \ref{thm:main_theorem} is a direct consequence of the uniform integrability assumption. That (iii) implies (i) follows from a generalization of Theorem 2.6 in \cite{PeccatiThaeleGamma} stated as Proposition \ref{prop: Theorem 2.6} below. Showing the implication (ii) $\Longrightarrow$ (iii) is the main part of the proof. While the proof of the corresponding implication in Theorem \ref{thm:PoissonNormal} is rather straight forward and works for arbitrary $q\geq 2$, the proof here is based on a couple of new estimates and arguments. They are of independent interest and might also be helpful beyond the context of the present paper. In sharp contrast to Theorem \ref{thm:PoissonNormal}, our arguments show that the `usual' technique (relying on the multiplication formula for Poisson integrals similar as in the proofs of Theorems \ref{thm:GaussianNormal}, \ref{thm:Gaussian-Gamma} or \ref{thm:PoissonNormal}) for proving the implication (ii) $\Longrightarrow$ (iii) only works in case that $q=2$ and $q=4$ and cannot be improved. The main reason for this is the involved combinatorial structure on a Poisson chaos implied by the multiplication formula \eqref{eqn:multiplication}. The proof of Theorem \ref{thm:main_theorem} is the content of Section \ref{sec:ProofMainThm} below.

Theorem \ref{thm:main_theorem} has a counterpart in a free probability setting, see \cite{Bourgin2013}. Here, one studies the approximation of the law of a sequence of elements belonging to a fixed chaos of order $q\geq 1$ of the so-called \textit{free} Poisson algebra by the Marchenko-Pastur law (also called free Poisson law). It is interesting to see that in this case, the proof works for arbitrary $q\ge1$ and does not need a sign condition on the kernels. This is explained by the relatively simple combinatorial structure on a free Poisson chaos, which is inherited from the free multiplication formula in which all combinatorial coefficients are equal to one. This causes that the expressions for the third and fourth moment are much simpler compared to the classical set-up of the present paper and implies that a proof of the corresponding free four moments theorem works in full generality.

Comparing Theorem \ref{thm:PoissonNormal} and Theorem \ref{thm:main_theorem}, it is natural to ask whether there exists a version of Theorem \ref{thm:main_theorem} dealing with a sequence of non-negative kernels. Indeed, Corollary \ref{cor:NonnegativeKernels} below provides such a version, but it deals with a different limiting law, namely what we call a centred reflected Gamma distribution. In case of a limiting Gaussian law, this phenomonon is not visible, since a Gaussian law is symmetric, see also the discussion in Remark \ref{remark:symmetry}.

\begin{definition}[Centred reflected Gamma distribution]
A random variable $Y$ has a centred reflected Gamma distribution $\widehat\G_\nu$ with parameter $\nu>0$, if $-Y\sim\overline \G_\nu$.
\end{definition}

Note that if $Y\sim\widehat \G_\nu$ follows a centred reflected Gamma distribution with parameter $\nu$, the first four moments of $Y$ are given by
\begin{align*}
\EE[Y] &= 0\,, \quad\EE[Y^2] = 2\nu\,, \quad\EE[Y^3] = -8\nu\,, \quad
\EE[Y^4] = 12\nu^2 + 48\nu\,.
\end{align*}
Moreover, while the centred Gamma distribution has support $[-\nu,\infty)$, the centred reflected Gamma distribution is supported on $(-\infty,\nu]$. The next result is an immediate consequence of Theorem \ref{thm:main_theorem} and the definition of $\widehat \G_\nu$.

\begin{corollary}[Four moments theorem for Poisson integrals with non-negative kernels] \label{cor:NonnegativeKernels}
Fix $\nu>0$. Let $q\geq 2$ be an even integer and $\{f_n\colon n\ge1\} \subset L_s^2(\mu_n^q)$ be a sequence of kernels satisfying the technical assumptions (A) and the normalization condition
\[
\lim_{n\to\infty} q!\|f_n\|^2 = \lim_{n\to\infty}\EE[I_q(f_n)^2] = 2\nu\,.
\]
Let the sequence $\{I_q^4(f_n)\colon n\ge1\}$ be uniformly integrable and suppose that $Y\sim\widehat\G_\nu$ is a random variable having a centred reflected Gamma distribution with parameter $\nu$. If one of the conditions
\begin{itemize}
\item[(a)] $q=2$ and $\lim\limits_{n\to\infty}\|f_n^2\|=0$,
\item[(b)] $q=4$ and $f_n\geq 0$ for all $n\geq 1$
\end{itemize}
is satisfied, then the following three assertions are equivalent:
\begin{enumerate}[(i)]
\item
As $n\to\infty$, the sequence $\{I_q(f_n)\colon n\ge1\}$ converges in distribution to $Y$;
\item
$\lim\limits_{n\to\infty} \EE[I_q(f_n)^4]  + 12\EE[I_q(f_n)^3]= 12\nu^2 - 48\nu$;
\item
$\lim\limits_{n\to\infty}\|f_n \star_r^\ell f_n\|=0$ for all $r\in \{1,\ldots, q\}$, $\ell \in\{1, \ldots, r \wedge (q-1)\}$ such that $(r,\ell)\neq (q/2,q/2)$,
$\lim\limits_{n\to\infty}\|f_n^2 \|=0$, and 
$\lim\limits_{n\to\infty}\|f_n \,\tld\star_{q/2}^{q/2} f_n + c_q \,f_n\| = 0$ with $c_q = \frac{4}{(q/2)! \binom{q}{q/2}^2}$.
\end{enumerate}
\end{corollary}

\begin{remark}
We emphasize that one could derive our main result, Theorem \ref{thm:main_theorem}, also for the \textit{two-parametric} centred Gamma distribution $\overline \Gamma_{a,\lambda}$, $a,\lambda>0$, with probability density
\[
h_{a,\lambda}(x) = \frac{\lambda^{a}}{\G(a)}\big(x+\tfrac{a}{\lambda}\big)^{a - 1} e^{-(\lambda x+a)}\, \one\big(x>-\tfrac{a}{\lambda}\big).
\]
The one-parametric centred Gamma distribution $\overline \G_\nu$ then arises by putting $a=\tfrac{\nu}{2}$ and $\lambda=\tfrac12$. In order to allow for a better comparison with the existing literature \cite{NourdinPeccati2009a,PeccatiThaeleGamma} and to keep the presentation transparent, we have decided to restrict to the one-parametric case.
\end{remark}

\section{Application to homogeneous sums and $U$-statistics}\label{sec:applications}

\subsection{Homogeneous sums}

According to \cite{PeccatiZheng2013} a universality result is a `mathematical statement implying that the asymptotic behaviour of a large random system does not depend on the distribution of its components'. Such results are at the heart of modern probability and the class of examples comprises the classical central limit theorem or the semicircular law in free probability. In this section, we shall derive a universality result for so-called homogeneous sums based on a sequence of independent centred Poisson random variables. For further background material concerning universality results for homogeneous sums we refer to the monograph \cite{Nourdin2012Book} as well as to the original papers \cite{NourdinPeccatiReinert2010, PeccatiZheng2013}.

We start by introducing the notion of a particularly well-behaved class of kernels.

\begin{definition}[Index functions]\label{def:index functions}
Fix an integer $q\ge1$. A function $h\colon \NN^q \to \RR$ is an \emph{index function} of order $q$, if
\begin{enumerate}
\item
$h$ is symmetric meaning that $h(i_1,\ldots, i_q) = h(i_{\pi(1)}, \ldots, i_{\pi(q)})$ for all $(i_1,\ldots, i_q)\in \NN^q$ and all permutations $\pi\in\Pi_q$;
\item
it vanishes on diagonals meaning that for $(i_1,\ldots, i_q)\in \NN^q$, $h(i_1,\ldots, i_q) =0$ whenever $i_k = i_\ell$ for some $k\neq \ell$.
\end{enumerate}
Fix an integer $N\ge1$. If $g$ and $h$ are two index functions of order $q$, we define their scalar product by
\begin{align*}
\lan g,h \ran_{(N,q)} &= \sum_{1\le i_1,\ldots, i_q\le N} g(i_1,\ldots, i_q) h(i_1,\ldots, i_q) %\\
%q! \sum_{\{i_1,\ldots, i_q\} \subset [N]^q} g(i_1,\ldots, i_q) h(i_1,\ldots, i_q)
\end{align*}
and write $\|h\|_{(N,q)} = \lan h,h \ran_{(N,q)}^{1/2}$ for the corresponding norm. We frequently suppress the subscript $(N,q)$ if it is clear from the context.
\end{definition}

As in Section \ref{sec:4Moments}, we denote by $\{\mu_n:n\geq 1\}$ a sequence of $\s$-finite non-atomic measures on some Polish space $(\cZ,\sZ)$.

\begin{definition}[Tamed sequences]\label{def:tamed sequences}
Fix an integer $q\ge1$. A sequence $\{f_n\colon n\ge1\}\subset L_s^2(\mu_n^q)$ is \emph{tamed} if there exists a sequence of integers $\{N_n\colon n\ge1\}$ with $N_n\to\infty$, as $n\to\infty$, and an infinite measurable partition $\{B_i\colon i\ge1\}$ of $\cZ$ verifying the following conditions:
\begin{enumerate}
\item
there exists $\a\in(0,\infty)$ such that $\a<\mu_n(B_i)<\infty$ for every $i,n\ge1$,
\item
there is a sequence of index functions $\{h_n\colon n\ge1\}$ of order $q$, such that $f_n$ has the representation 
\be{eqn:tamed}
f_n(z_1, \ldots, z_q) =\sum_{1\le i_1,\ldots, i_q\le N_n} h_n(i_1,\ldots, i_q)\,\prod_{k=1}^q{\one_{B_{i_k}} (z_k)\over \sqrt{\mu_n(B_{i_k})}}\,.
\ee
%where for $k\in\{1,\ldots, q\}$, $g_{i_k} (z_k) =\one_{B_{i_k}} (z_k)/\sqrt{\mu_n(B_{i_k})}$.
\end{enumerate}
\end{definition}

\begin{remark}
\begin{enumerate}
\item
It follows from the definition that if a sequence $\{f_n\colon n\ge1\}\subset L_s^2(\mu_n^q)$ is tamed, we necessarily must have that $\mu_n(\cZ)=\infty$ for every $n\ge1$.
\item 
If $\{f_n\colon n\ge1\}\subset L_s^2(\mu_n^q)$ is a tamed sequence with a representation as at \eqref{eqn:tamed}, we have that $ \|h_n\|_{(N_n,q)} = \|f_n\|_{L^2(\mu_n^q)} <\infty$.
\item
One easily verifies that tamed sequences automa\-tically satisfy the technical assumptions (A).
\end{enumerate}
\end{remark}

\begin{definition}[Homogeneous sums]\label{def:homogeneous sums}
Fix integers $N,q\ge1$ and let $\bX = \{X_i\colon i\ge1\}$ be a sequence of random variables. Let $h$ be an index function of order $q$. Then
\[
Q_q(N,h,\bX) = \sum_{1\le i_1,\ldots, i_q\le N} h(i_1,\ldots, i_q) X_{i_1}\cdots X_{i_q}
\]
is the \emph{homogeneous sum} of $h$ of order $q$ based on the first $N$ elements of $\bX$.
\end{definition}

If $\bX = \{X_i\colon i\ge1\}$ is a sequence of independent and centred random variables with unit variance, then
\begin{align*}
\EE[Q_q(N,h,\bX)]=0, && \EE[Q_q(N,h,\bX)^2] = q! \| h \|_{(N,q)}^2.
\end{align*}

In what follows, two particular classes of random variables play a special role. By $\bG = \{G_i\colon i\ge1\}$ we indicate a sequence of independent and identically distributed random variables, such that $G_i\sim\cN(0,1)$ for every $i\geq 1$. Moreover, we shall write $\bP = \{P_i\colon i\ge1\}$ for a sequence of independent random variables verifying
\[
P_i \stackrel{d}{=} \frac{{\rm Po}(\l_i) - \l_i}{\sqrt{\l_i}}\,, \qquad i\ge1\,,
\]
where ${\rm Po}(\l_i)$ indicates a Poisson random variable with mean $\l_i$, such that $\a=\inf\{\l_i:i\geq 1\}>0$.

There is a close connection between homogeneous sums based on $\bP$ (or $\bG$) and multiple stochastic integrals with respect to a centred Poisson measure $\hat \n_n$ (or a Gaussian measure $G_n$) of tamed sequences. Namely, if $q\geq 1$ is a fixed integer and $\{f_n\colon n\ge1\}\subset L_s^2(\mu_n^q)$ is a tamed sequence with representation \eqref{eqn:tamed}, then there is a sequence of centred Poisson measures $\{\hat\eta_n\colon n\ge1\}$ (or a sequence of Gaussian measures $\{G_n\colon n\ge1\}$) such that
\begin{align}\label{eqn:correspondence}
I_q(f_n) = Q_q(N_n, h_n, \bP), && I_q^{G_n}(f_n) = Q_q(N_n, h_n, \bG).
\end{align}
Vice versa, given a sequence of index functions $\{h_n\colon n\ge1\}$ of order $q\ge1$ and a sequence of integers $\{N_n\colon n\ge1\}$ diverging to infinity, as $n\to\infty$, such that $\|h_n\|_{(N_n,q)}<\infty$ for every $n\ge1$, then there is a tamed sequence $\{f_n\colon n\ge1\}$ with representation \eqref{eqn:tamed} and sequences of centred Poisson measures $\{\hat\eta_n\colon n\ge1\}$ and Gaussian measures $\{G_n\colon n\ge1\}$ such that \eqref{eqn:correspondence} holds.

The following result is a version of \cite[Theorem 1.8]{NourdinPeccatiReinert2010} and \cite[Theorem 1.12]{NourdinPeccatiReinert2010}. Notice that there, the results are stated for integer-valued parameters $\nu\ge1$, but they continue to hold for any $\nu>0$. 

\begin{theorem}[Gamma universality of homogeneous sums on a fixed Gaussian chaos]\label{thm:Gaussian universality}
Fix $\nu>0$, let $q\ge2$ be even and $\{f_n\colon n\ge1\}\subset L^2(\mu_n^q)$ be a tamed sequence with representation \eqref{eqn:tamed} that satisfies the normalization condition
\[
\lim_{n\to\infty} q! \|f_n\|^2 = \lim_{n\to\infty} \EE[I_q^G(f_n)^2]= \lim_{n\to\infty} \EE[Q_q(N_n, h_n, \bG)^2] = 2\nu.
\]
Let $Y\sim \overline{\Gamma}_\nu$ be a centred Gamma random variable with parameter $\nu$. Then the following five assertions are equivalent:
\begin{enumerate}[(i)]
\item
As $n\to\infty$, the sequence $\{Q_q(N_n, h_n, \bG)\colon n\ge1\}$ converges in distribution to $Y$;
\item
$\lim\limits_{n\to\infty} \EE[Q_q(N_n, h_n, \bG)^4]  - 12\EE[Q_q(N_n, h_n, \bG)^3]= 12\nu^2 - 48\nu$;
\item
$\lim\limits_{n\to\infty}\|f_n \star_r^r f_n\|=0$ for every $r\in \{1,\ldots, q-1\}\setminus\{q/2\}$, and \\
$\lim\limits_{n\to\infty}\|f_n \,\tld\star_{q/2}^{q/2} f_n - c_q \,f_n\| = 0$ with $c_q = \frac{4}{(q/2)! \binom{q}{q/2}^2}$;
\item
for every sequence $\bX = \{X_i\colon i\ge1\}$ of independent centred random variables with unit variance which is such that $\sup_i\EE|X_i|^{2+\e} < \infty$ for some $\e>0$, the sequence $\{Q_q(N_n, h_n, \bX)\colon n\ge1\}$ converges in distribution to $Y$, as $n\to\infty$;
\item
for every sequence $\bX = \{X_i\colon i\ge1\}$ of i.i.d. centred random variables with unit variance, the sequence $\{Q_q(N_n, h_n, \bX)\colon n\ge1\}$ converges in distribution to $Y$, as $n\to\infty$.
\end{enumerate}
\end{theorem}

The following result answers the question whether Theorem \ref{thm:Gaussian universality} continues to hold if in (i) and (ii) the class $\bG$ is replaced by $\bP$. Due to the discussion in Section \ref{subsec:$4$th-moment theorems in the Poisson chaos}, we cannot avoid additional assumptions in the Poisson case. In particular, we have to assume that either $q=2$ or $q=4$.

\begin{theorem}[Gamma universality of homogeneous sums on a fixed Poisson chaos]\label{thm:universality Poisson}
Fix $\nu>0$ and let $q\geq 2$ be even and $\{f_n\colon n\ge1\}\subset L^2(\mu_n^q)$ be a tamed sequence with representation \eqref{eqn:tamed} that satisfies the normalization condition
\be{eqn:normalization universality}
\lim_{n\to\infty} q! \|f_n\|^2 = \lim_{n\to\infty} \EE[I_q(f_n)^2]= \lim_{n\to\infty} \EE[Q_q(N_n, h_n, \bP)^2] = 2\nu.
\ee
Let $Y\sim\overline\G_\nu$ be a random variable following a centred Gamma distribution with parameter $\nu$. If one of the conditions
\begin{enumerate}
\item
$q=2$ and $\lim\limits_{n\to\infty} \|f_n^2\|=0$,
\item
$q=4$ and $f_n\le 0$ for all $n\ge1$
\end{enumerate}
is satisfied, then the following five assertions are equivalent:
\begin{enumerate}[(i)]
\item
As $n\to\infty$, the sequence $\{Q_q(N_n, h_n, \bP)\colon n\ge1\}$ converges in distribution to $Y$;
\item
$\lim\limits_{n\to\infty} \EE[Q_q(N_n, h_n, \bP)^4]  - 12\EE[Q_q(N_n, h_n, \bP)^3]= 12\nu^2 - 48\nu$;
\item
$\lim\limits_{n\to\infty}\|f_n \star_r^r f_n\|=0$ for all $r\in \{1,\ldots, q-1\}\setminus\{q/2\}$, and \\ 
$\lim\limits_{n\to\infty}\|f_n \,\tld\star_{q/2}^{q/2} f_n - c_q \,f_n\| = 0$ with $c_q = \frac{4}{(q/2)! \binom{q}{q/2}^2}$;
\item
for every sequence $\bX = \{X_i\colon i\ge1\}$ of independent centred random variables with unit variance which is such that $\sup_i\EE|X_i|^{2+\e} < \infty$ for some $\e>0$, the sequence $\{Q_q(N_n, h_n, \bX)\colon n\ge1\}$ converges in distribution to $Y$, as $n\to\infty$;
\item
for every sequence $\bX = \{X_i\colon i\ge1\}$ of i.i.d. centred random variables with unit variance, the sequence $\{Q_q(N_n, h_n, \bX)\colon n\ge1\}$ converges in distribution to $Y$, as $n\to\infty$.
\end{enumerate}
\end{theorem}

\begin{proof}
At first, we observe that due to Theorem \ref{thm:Gaussian universality}, the assertions (iii), (iv) and (v) are equivalent. In \cite[Subsection 4.2]{PeccatiZheng2013}, it has been argued that
\be{eqn:bounded}
\sup_{i\ge1} \EE|P_i|^p <\infty
\ee 
for all $p\geq 1$. This means that $\bP$ is a special instance of a sequence with the properties in assertion (iv) such that we obtain the implication (iv) $\Longrightarrow$ (i). Moreover, \eqref{eqn:bounded} implies together with the normalization condition \eqref{eqn:normalization universality} and \cite[Lemma 4.2]{NourdinPeccatiReinert2010} that the sequence $\{Q_q(N_n, h_n, \bP)^4\colon n\ge1\}$ is uniformly integrable such that we get the implication (i) $\Longrightarrow$ (ii).

To prove (ii) $\Longrightarrow$ (iii), we apply Theorem \ref{thm:main_theorem}. For this, one has to observe that assertion (iii) in Theorem \ref{thm:main_theorem} implies assertion (iii) in Theorem \ref{thm:universality Poisson}.
\end{proof}

\begin{remark}
Theorem \ref{thm:universality Poisson} shows that one can dispense with the assumption on the uniform integrability of the sequence $\{I_q(f_n)^4\colon n\ge1\}$ in Theorem \ref{thm:main_theorem} whenever the sequence $\{f_n\colon n\ge1\} \subset L_s^2(\mu_n^q)$ is tamed.
\end{remark}

\begin{remark}
Replacing in (b) the condition that $f_n\leq 0$ by $f_n\geq 0$, in (ii) the moment condition by $\lim\limits_{n\to\infty} \EE[Q_q(N_n, h_n, \bP)^4]  + 12\EE[Q_q(N_n, h_n, \bP)^3]= 12\nu^2 - 48\nu$ and in (iii) the condition on the middle-contraction by $\|f_n \,\tld\star_{q/2}^{q/2} f_n + c_q \,f_n\|\to 0$, one arrives at a version of Theorem \ref{thm:universality Poisson} with a centred reflected limiting random variable $Y\sim\widehat{\Gamma}_\nu$ in assertion (i), (iv) and (v).
\end{remark}

\subsection{$U$-statistics}

Our second application is concerned with $U$-statistics. To introduce them, fix an integer $d\geq 1$ and let $\bY=\{Y_i:i\geq 1\}$ be a sequence of i.i.d.\ random vectors in $\RR^d$, whose distribution has a density $p(\cdot)$ with respect to the Lebesgue measure on $\RR^d$. Next, for any $n\geq 1$, let $N_n$ be a Poisson random variable with mean $n$ and define 
\be{eqn: eta_n}
\eta_n=\sum_{i=1}^{N_n}\delta_{Y_i}\,.
\ee 
Clearly, $\eta_n$ is a Poisson random measure on $\RR^d$ with control measure $\mu_n(\dint x)=np(x)\,\dint x$, implying that $\mu_n(\RR^d)=n\to\infty$, as $n\to\infty$. Now, we put $\hat\eta_n= \eta_n - \mu_n$ and set $\mu=\mu_1$ for the sake of convenience. By a \textit{Poisson $U$-statistic} of order $q\geq 2$ based on $\eta_n$ we mean in this paper a random variable of the form
\[
U_n=\sum_{1\le i_1<\cdots<i_q\le N_n}h_n(Y_{i_1},\ldots,Y_{i_q})\,,\qquad n\geq 1\,,
\]
where the kernel $h_n\colon (\RR^d)^q\to\RR$ is an element of $L_s^1(\mu^q)$. On the other hand, a \textit{classical $U$-statistic} is a random variable $\hat{U}_n$ such that
\[
\hat{U}_n=\sum_{1\le i_1<\cdots<i_q\le n}h_n(Y_{i_1},\ldots,Y_{i_q})\,,\qquad n\geq 1\,.
\]

The difference between $U_n$ and $\hat{U}_n$ is that $U_n$ involves a random number $\binom{N_n}{q}$ of summands, while the number of summands in the definition of $\hat{U}_n$ is fixed (namely $\binom{n}{q}$). We say that a (Poisson or classical) $U$-statistic is \textit{completely degenerate} if 
\[
\int_{\RR^d}h_n(x,z_1,\ldots,z_{q-1})\,p(x)\,\dint x=0
\] 
for $\mu^{q-1}$-almost every $(z_1,\ldots,z_{q-1})\in (\RR^d)^{q-1}$. In particular, this implies that $\EE[U_n]=\EE[\hat{U}_n]=0$. Moreover, we suppose that $U_n$ and $\hat{U}_n$ are square-integrable.

We recall the following particular case of a celebrated theorem of de Jong, which provides a simple moment condition under which a central limit theorem for a sequence of completely degenerate $U$-statistics is guaranteed.

\begin{theorem}[de Jong \cite{deJong,deJong2}]\label{thm:deJongNormal}
Let $q\geq 2$ and $\{h_n : n\geq 1\}$  be a sequence of non-zero elements of  $L^4_s(\mu^q)$. Suppose that the $U$-statistics $U_n$ and $\hat{U}_n$ are completely degenerate and define $\sigma^2(n)= {\rm Var}( U_n)$. Then the moment condition
$\lim\limits_{n\to\infty}\frac{\EE[U_n^4]}{\sigma(n)^4} = 0$ implies that, as $n\to \infty$, the sequences $U_n/\sigma(n)$ and $\hat{U}_n/\sigma(n)$ converge in distribution to a standard Gaussian random variable.
\end{theorem}

In our paper, we are interested in the Gamma approximation of Poisson and classical $U$-statistics. The next result generalizes Theorem 2.13 (B) in \cite{PeccatiThaeleGamma}, where the authors had to restrict to the case $q=2$. Here, we add a corresponding limit theorem in case that $q=4$ under an additional sign condition. It can be seen as a non-central version of de Jong's theorem, Theorem \ref{thm:deJongNormal}. We shall see that in the non-central case a similar result is true under a suitable condition involving only the third and the fourth moment.

\begin{theorem}\label{thm:deJong}
Suppose that $q\in\{2,4\}$. For each $n\geq 1$ let $h_n\in L_s^4(\mu^q)$ be a function such that $$\sup_{n\geq 1}{\int h_n^4\,\dint\mu_n^q\over(\int h_n^2\,\dint\mu_n^q)^2}<\infty$$ and suppose that the $U$-statistics $U_n$ and $\hat{U}_n$ are completely degenerate. Further assume that there exists $\nu>0$ such that $\lim\limits_{n\to\infty}\EE[U_n^2]=2\nu$ and that
\begin{itemize}
\item[(a)] $\lim\limits_{n\to\infty}\|h_n^2\|=0$ if $q=2$,
\item[(b)] $f_n\leq 0$ for all $n\ge1$ if $q=4$.
\end{itemize}
Then the moment condition $\lim\limits_{n\to\infty}\EE[U_n^4]-12\EE[U_n^3]=12\nu^2-48\nu$ implies that both random variables $U_n$ and $\hat{U}_n$ converge in distribution to $Y\sim \overline \Gamma_\nu$, as $n\to\infty$.
\end{theorem}
\begin{proof}
Using the fact that the Poisson $U$-statistics $U_n$ is an element of the sum of the first $q$ Poisson chaoses with respect to $\hat \eta_n$ as introduced after \eqref{eqn: eta_n} (see \cite[Theorem 3.6]{ReitznerSchulte2013}), as well as the fact that $U_n$ is completely degenerate, one obtains that $U_n=I_q(h_n)$ for every $n\geq 1$. The result for the Poisson $U$-statistics $U_n$ then follows immediately from Theorem \ref{thm:main_theorem}. Moreover, it is known from \cite{DynkinMandelbaum} that $\EE[(U_n-\hat{U}_n)^2]=O(n^{-1/2})$, as $n\to\infty$. This yields the result also for $\hat{U}_n$.
\end{proof}

\begin{remark}
Using Theorem 2.6 in \cite{PeccatiThaeleGamma} or its generalization Proposition \ref{prop: Theorem 2.6} below, one can add a rate of convergence (for a certain smooth probability distance) between $U_n$ or $\hat{U}_n$ and the limiting random variable $Y$. However, we do not pursue such quantitative results in this paper.
\end{remark}

\begin{remark}
In assumption (b) of Theorem \ref{thm:deJong} one can replace the sign condition $f_n\leq 0$ by $f_n\geq 0$ and at the same time the moment condition $\EE[U_n^4]-12\EE[U_n^3]\to 12\nu^2-48\nu$ by $\EE[U_n^4]+12\EE[U_n^3]\to 12\nu^2-48\nu$. In this case, the limiting random variable $Y$ has a centred reflected Gamma distribution $\widehat{\Gamma}_\nu$ with parameter $\nu>0$.
\end{remark}

\section{Proof of Theorem \ref{thm:main_theorem}}\label{sec:ProofMainThm}

\subsection{Strategy of the proof}

Before entering the details of the proof of Theorem \ref{thm:main_theorem}, let us briefly summarize the overall strategy.

First of all, the implication (i) $\Longrightarrow$ (ii) of Theorem \ref{thm:main_theorem} is a direct consequence of the uniform integrability of the sequence $\{I_q(f_n)^4\colon n\ge1\}$. Next, the implication (iii) $\Longrightarrow$ (i) will follow from a generalization of the main result of \cite{PeccatiThaeleGamma}, which has been derived by the Malliavin-Stein method. It delivers a criterion in terms of contraction norms, which ensures centred Gamma convergence on a fixed Poisson chaos of even order and is presented as Proposition \ref{prop: Theorem 2.6} below. The main part of proof of Theorem \ref{thm:main_theorem} consists in showing that (ii) implies (iii). It is based on the technical Lemmas \ref{lemma:symmetrization} and \ref{lemma:reverse inequality}, which establish new inequalities for norms of contraction kernels, that are also of independent interest. Next, in Lemma \ref{lemma:A+R} we derive an asymptotic lower bound for the moment expression $ \EE[I_q(f_n)^4]  - 12\EE[I_q(f_n)^3]$ in terms of contraction norms. Finally, Lemma \ref{lemma:A>=0} shows under the conditions of Theorem \ref{thm:main_theorem} that if the lower bound for $ \EE[I_q(f_n)^4]  - 12\EE[I_q(f_n)^3]$ converges to the `correct' quantity, the contraction conditions in (iii) are satisfied. Lemma \ref{lemma:A to 0} proves that this lower bound actually converges.

We emphasize that we state all intermediate steps of the proof of Theorem \ref{thm:main_theorem} as general as possible in order to highlight in which step the restrictive condition that $q=2$ or $q=4$ and the sign condition on the kernels arise. 

\subsection{Preparatory steps}

We start our investigations with a generalization of Theorem 2.6 in \cite{PeccatiThaeleGamma}. The main difference between that result and Proposition \ref{prop: Theorem 2.6} is that for technical reasons it has been assumed in \cite{PeccatiThaeleGamma} that $\mu_n$ is a finite measure for each $n\geq 1$ such that $\mu_n(\cZ)\to\infty$, as $n\to\infty$. Our next result shows that one can dispense with this assumption.

\begin{proposition}\label{prop: Theorem 2.6}
Fix $\nu>0$ and an even integer $q\ge2$. Let the sequence $\{f_n\colon n\ge1\} \subset L_s^2(\mu_n^q)$ satisfy the technical assumptions (A) and the normalization condition
\[
\lim_{n\to\infty} q! \|f_n\|^2 = \lim_{n\to\infty}\EE[I_q(f_n)^2] = 2\nu\,.
\]
Then, if
\begin{align*}
&\lim_{n\to\infty}\|f_n \star_r^\ell f_n\|=0  \text{ for all } r\in \{1,\ldots, q\}\text{ and }\ell \in\{1, \ldots, r \wedge (q-1)\}, \
	(r,\ell)\neq (q/2,q/2)\,,\\
&\lim_{n\to\infty}\|f_n^2 \|=0\,, \\
&\lim_{n\to\infty}\|f_n \,\tld\star_{q/2}^{q/2} f_n - c_q \,f_n\| = 0 \text{ with }c_q = \frac{4}{(q/2)! \binom{q}{q/2}^2}\,,
\end{align*}
the sequence $\{I_q(f_n)\colon n\ge1\}$ converges in distribution to $Y\sim\overline \G_\nu$, as $n\to\infty$.
\end{proposition}

\begin{proof}
In principle, one can follow the proof of \cite[Theorem 2.6]{PeccatiThaeleGamma}. The only part where the assumption about the finiteness of the measures $\mu_n$ enters is \cite[Proposition 2.3]{PeccatiThaeleGamma}. To circumvent this problem, one uses the modified integration-by-parts formula \cite[Lemma 2.3]{Schulte2012} and concludes as in the proof of Theorem 4.1 of \cite{EichelsbacherThaele2013}. Since the computations are quite straight forward, we omit the details.
\end{proof}

We now present two estimates of the norm of a symmetrized contraction kernel in terms of non-symmetrized contraction norms. In particular, our first lemma generalizes \cite[Identity (11.6.30)]{PeccatiTaqquBook}. We recall for $f\in L_s^2(\mu^q)$, $q\ge1$, that $\| f\,\tld \star_q^q f\|^2 = \| f \star_q^q f\|^2 = \|f\|^4$ and $\|f\star_0^0 f\|^2 = \|f\|^4$.

\begin{lemma}\label{lemma:symmetrization}
Let $q\ge1$ be an integer and $f\in L_s^2(\mu^q)$ be a kernel satisfying the technical assumptions (A). Then 
\be{eqn:0-contraction}
\|f \,\tld\star_0^0 f\|^2 = \frac{(q!)^2}{(2q)!} \bigg( 2 \|f\|^4 + \sum_{p=1}^{q-1} \binom{q}{p}^2 \|f \star_p^p f\|^2 \bigg).
\ee
Furthermore, for any $r\in\{1, \ldots, q-1\}$ one has the inequality
\be{eqn:r-contraction}
\|f \,\tld\star_r^r f\|^2 \le \frac{((q-r)!)^2}{(2(q-r))!} \bigg( 2 \|f \star_r^r f\|^2 + \sum_{p=1}^{q - r -1} \binom{q-r}{p}^2 \|f \star_p^p f\|^2 \bigg).
\ee
\end{lemma}

If $q\geq 2$ is an even integer, Equation \eqref{eqn:r-contraction} yields that
\be{eqn:q/2-contraction}
\|f \,\tld\star_{q/2}^{q/2} f\|^2 \le \frac{((q/2)!)^2}{q!} \bigg( 2 \|f\star_{q/2}^{q/2} f\|^2 + \sum_{p=1}^{q/2 -1} \binom{q/2}{p}^2 \|f \star_p^p f\|^2 \bigg).
\ee
This inequality will turn out to be crucial in what follows.

Before entering the proof of Lemma \ref{lemma:symmetrization}, we introduce some notation. Recall that for an integer $p\ge1$, we denote the group of $p!$ permutations of the set $\{1,\ldots, p\}$ by $\Pi_p$. For a kernel $g\in L^2(\mu^p)$ and a permutation $\pi\in\Pi_p$, we use the shorthand $g(\pi)$ for the mapping $\cZ^p \ni (z_1,\ldots, z_p) \mapsto g(\pi )(z_1,\ldots, z_p) = g(z_{\pi(1)}, \ldots, z_{\pi(p)}) $. We can immediately see that $\|g\| = \|g (\pi) \|$ for all $\pi\in\Pi_p$ such that automatically $g( \pi) \in L^2(\mu^p)$. In the following, we use the convention that $\pi_0\in\Pi_p$ is the identity map, meaning that $g(\pi_0) = g$.

For any integer $M\ge1$, any two permutations $\pi,\sigma \in \Pi_{2M}$ and any $p\in \{0,\ldots , M\}$ we shall use the notation 
\[
\pi \sim_p \s
\]
if and only if 
\[
|\{\pi(1),\ldots, \pi(M)\} \cap \{\sigma(1), \ldots, \sigma(M)\}| = p\,,
\]
where $|\,\cdot\,|$ stands for the cardinality of the argument set.
If $\pi \sim_p \s$, then clearly $|\{\pi(M+1),\ldots, \pi(2M)\} \cap \{\sigma(M+1), \ldots, \sigma(2M)\}| = p$. In the proof of \cite[Proposition 11.2.2]{PeccatiTaqquBook}, there is an explanation that, given a permutation $\pi\in\Pi_{2M}$ and an integer $p\in \{0,\ldots , M\}$, there are exactly $(M!)^2\binom{M}{p}^2$ permutations $\s\in\Pi_{2M}$ such that $\pi \sim_p \s$.

\begin{proof}[Proof of Lemma \ref{lemma:symmetrization}]
Let $q\ge1$ be an integer and $f\in L_s^2(\mu^q)$ be a kernel satisfying the technical assumptions (A). Fix $r\in \{0,1,\ldots, q-1\}$. We have that 
\begin{align}
\begin{split} \label{eqn:norm-expansion}
\|f\,\tld\star_r^r f\|^2 = \lan f \star_r^r f, f \, \tld \star_r^r f\ran
&= \frac{1}{(2q - 2r)!} \sum_{\pi \in \Pi_{2q-2r}} \lan f \star_r^r f, f \star_r^r f (\pi)\ran \\
&=\frac{1}{(2q - 2r)!} \sum_{p=0}^{q-r} \sum_{\pi \sim_p \pi_0} \lan f \star_r^r f, f \star_r^r f (\pi)\ran.
\end{split}
\end{align}

To prove \eqref{eqn:0-contraction}, let $r=0$ and $\pi \sim_0 \pi_0$ or $\pi\sim_q \pi_0$. Then we get 
\begin{align*}
&\lan f \star_0^0 f, f \star_0^0 f (\pi)\ran \\
&= \int_{\cZ^{2q}} f(z_1, \ldots, z_{q}) f( z_{q+1}, \ldots, z_{2q})
	f(z_{\pi(1)}, \ldots, z_{\pi(q)}) f( z_{\pi(q+1)}, \ldots, z_{\pi(2q)}) \mu^{2q} (\dint(z_1 \ldots z_{2q}))  \\
&= \Big( \int_{\cZ^q} f(w_1, \ldots, w_q)^2 \mu^q(\dint( w_1 \ldots w_q))\Big)^2 \\
&= \|f\|^4.
\end{align*}

Now, let $\pi \sim_p \pi_0$ with $p\in\{1, \ldots, q-1\}$. Then
\begin{align*}
&\lan f \star_0^0 f, f \star_0^0 f (\pi)\ran \\
&= \int_{\cZ^{2q}} f(z_1, \ldots, z_{q}) f( z_{q+1}, \ldots, z_{2q})
	f(z_{\pi(1)}, \ldots, z_{\pi(q)}) f( z_{\pi(q+1)}, \ldots, z_{\pi(2q)}) \mu^{2q} (\dint (z_1 \ldots z_{2q}))  \\
&= \int_{\cZ^{2q-2p}\times\cZ^p\times\cZ^p}f(z_1, \ldots, z_{q}) f(z_{\pi(1)}, \ldots, z_{\pi(q)})  \\
	&\qquad\qquad\qquad\qquad \times  f( z_{q+1}, \ldots, z_{2q}) f( z_{\pi(q+1)}, \ldots, z_{\pi(2q)}) 
	\mu^{2q}(\dint( z_1 \ldots z_{2q}))  \\
&\stackrel{(\star)}{=}  \int_{\cZ^{2q-2p}} f \star_p^p f (w_1, \ldots, w_{2q-2p}) \times f \star_p^p f (w_{1}, \ldots, w_{2q - 2p})  
	\mu^{2q - 2p} (\dint (w_1 \ldots w_{2q - 2p}))\\
&= \|f \star_p^p f\|^2\,.
\end{align*}
We note that the assumption that $f$ is symmetric is essential to get the identity highlighted by ($\star$). In view of \eqref{eqn:norm-expansion}, we obtain
\begin{align*}
&\|f\,\tld \star_0^0 f \|^2 \\
&= \frac{1}{(2q)!} \Bigg( \sum_{\pi \sim_0 \pi_0} \lan f \star_0^0 f, f \star_0^0 f (\pi)\ran+ \sum_{\pi \sim_q \pi_0} \lan f \star_0^0 f, f \star_0^0 f (\pi)\ran+ \sum_{p=1}^{q-1} \sum_{\pi\sim_p \pi_0} \lan f \star_0^0 f, f \star_0^0 f (\pi)\ran\Bigg)\\
&=  \frac{1}{(2q)!} \Bigg( 2 (q!)^2 \|f\|^4 +  \sum_{p=1}^{q-1} (q!)^2 \binom{q}{p}^2 \|f\star_p^p f\|^2 \Bigg)\,,
\end{align*} 
such that \eqref{eqn:0-contraction} follows. Now, let $r\in\{1, \ldots, q-1\}$ and observe that for $\pi \sim_{q-r} \pi_0$ one has that
\begin{align*}
&\lan f \star_r^r f, f \star_r^r f (\pi)\ran \\
&= \int_{\cZ^{2q - 2r}} \Big( \int_{\cZ^r} f(x_1, \ldots, x_r, z_1, \ldots, z_{q-r}) f(x_1, \ldots, x_r, z_{q-r+1}, \ldots, z_{2q-2r}) \mu^{r} (\dint( x_1\ldots x_r))\Big) \\
&\quad \times  \Big( \int_{\cZ^r} f(y_1, \ldots, y_r, z_{\pi(1)}, \ldots, z_{\pi(q-r)}) f(y_1, \ldots, y_r, z_{\pi(q-r+1)}, \ldots, z_{\pi(2q-2r)}) \mu^{r} (\dint (y_1\ldots  y_r))\Big) \\
&\qquad\qquad\qquad\qquad\mu^{2q - 2r} (\dint (z_1 \ldots  z_{2q-2r})) \\
&= \int_{\cZ^{2q}}f(x_1, \ldots, x_r, z_1, \ldots, z_{q-r}) f(y_1, \ldots, y_r, z_{\pi(1)}, \ldots, z_{\pi(q-r)}) \\
&\qquad\qquad \times f(x_1, \ldots, x_r, z_{q-r+1}, \ldots, z_{2q-2r}) f(y_1, \ldots, y_r, z_{\pi(q-r+1)}, \ldots, z_{\pi(2q-2r)}) \\
&\qquad\qquad\qquad\qquad\mu^{2q} (\dint (x_1\ldots  x_r,  y_1\ldots  y_r, z_1 \ldots z_{2q-2r})) \\
&= \int_{\cZ^{2r}\times\cZ^{q-r}\times\cZ^{q-r}}  f(x_1, \ldots, x_r, z_1, \ldots, z_{q-r}) f(y_1, \ldots, y_r, z_{\pi(1)}, \ldots, z_{\pi(q-r)}) \\
	&\qquad\qquad \times f(x_1, \ldots, x_r, z_{q-r+1}, \ldots, z_{2q-2r}) f(y_1, \ldots, y_r, z_{\pi(q-r+1)}, \ldots, z_{\pi(2q-2r)}) \\
	&\qquad\qquad\qquad\qquad\mu^{2q} (\dint( x_1\ldots  x_r, y_1\ldots  y_r, z_1 \ldots  z_{2q-2r})) \\
&=  \int_{\cZ^{2r}} \Big(f \star_{q-r}^{q-r} f (x_1, \ldots, x_r, y_1, \ldots, y_r)\Big)^2 \mu^{2r}(\dint (x_1\ldots x_r , y_1\ldots  y_r)) \\
&= \| f\star_{q-r}^{q-r} f\|^2 \\
&= \|f\star_r^r f\|^2\,.
\end{align*}

Similarly, we obtain for the case that $\pi \sim_0 \pi_0$,
\begin{align*}
\lan f \star_r^r f, f \star_r^r f (\pi)\ran 
%&= \int_{\cZ^{2r}} \Big[ \int_{\cZ^{q-r}} f(x_1, \ldots, x_r, z_1, \ldots, z_{q-r}) f(y_1, \ldots, y_r, z_{\pi(q-r+1)}, \ldots, z_{\pi(2q-2r)}) \Big] \\
%	&\quad \times \Big[ \int_{\cZ^{q-r}} f(x_1, \ldots, x_r, z_{q-r+1}, \ldots, z_{2q-2r}) f(y_1, \ldots, y_r, z_{\pi(1)}, \ldots, z_{\pi(q-r)}) \Big] \\
%	&\quad\mu^{2q} (\dint x_1\ldots \dint x_r \dint y_1\ldots \dint y_r\dint z_1 \ldots \dint z_{2q-2r}) \\
%&=  \int_{\cZ^{2r}} \Big[f \star_{q-r}^{q-r} f (x_1, \ldots, x_r, y_1, \ldots, y_r)\Big]^2 \mu^{2r}(\dint x_1\ldots \dint x_r \dint y_1\ldots \dint y_r) \\
%&= \| f\star_{q-r}^{q-r} f\|^2 \\
&= \|f\star_r^r f\|^2\,.
\end{align*}

Now, let $\pi \sim_p \pi_0$ with $p \in \{1, \ldots, q-r-1\}$. Then, there is a permutation $\s\in\Pi_{2q - 2p}$ such that
\begin{align}\label{eqn:computation}
&\lan f \star_r^r f, f \star_r^r f (\pi)\ran \\ \nonumber
&= \int_{\cZ^{2q}}f(x_1, \ldots, x_r, z_1, \ldots, z_{q-r}) f(y_1, \ldots, y_r, z_{\pi(1)}, \ldots, z_{\pi(q-r)}) \\ \nonumber
&\qquad\qquad \times f(x_1, \ldots, x_r, z_{q-r+1}, \ldots, z_{2q-2r}) f(y_1, \ldots, y_r, z_{\pi(q-r+1)}, \ldots, z_{\pi(2q-2r)}) \\ \nonumber
&\qquad\qquad\qquad\qquad\mu^{2q} (\dint( x_1\ldots  x_r,  y_1\ldots  y_r, z_1 \ldots  z_{2q-2r})) \\
\nonumber
&= \int_{\cZ^{2q - 2p}\times\cZ^p\times\cZ^p} f(x_1, \ldots, x_r, z_1, \ldots, z_{q-r}) f(y_1, \ldots, y_r, z_{\pi(1)}, \ldots, z_{\pi(q-r)}) \\ \nonumber
&\qquad\qquad \times  f(x_1, \ldots, x_r, z_{q-r+1}, \ldots, z_{2q-2r}) f(y_1, \ldots, y_r, z_{\pi(q-r+1)}, \ldots, z_{\pi(2q-2r)}) \\ \nonumber
&\qquad\qquad\qquad\qquad \mu^{2q} (\dint (x_1\ldots  x_r , y_1\ldots  y_r, z_1 \ldots  z_{2q-2r})) \\
\nonumber
&= \int_{\cZ^{2q - 2p}} f \star_p^p f(w_1, \ldots, w_{2q - 2p}) \times f \star_p^p f(w_{\s(1)}, \ldots, w_{\s(2q - 2p)})
	\mu^{2q - 2p} (\dint (w_1 \ldots  w_{2q-2p})) \\ \nonumber
&= \lan f\star_p^p f, f\star_p^p f(\s)\ran \\ \nonumber
&\le \| f\star_p^p f\| \, \| f\star_p^p f(\s)\| \\ \nonumber
&= \| f\star_p^p f\|^2\,.
\end{align}
Note that contrary to the case $r=0$, $\s$ shows up because of the appearance of the variables $x_1, \ldots, x_r, y_1, \ldots, y_r$. Therefore, we need to apply the Cauchy-Schwarz inequality once, which is the very reason for the inequality in \eqref{eqn:r-contraction}. At this stage, \eqref{eqn:r-contraction} follows by \eqref{eqn:norm-expansion} and
\begin{align*}
\|f\,\tld \star_r^r f \|^2
&= \frac{1}{(2q-2r)!} \left( \sum_{\pi \sim_0 \pi_0} \lan f \star_r^r f, f \star_r^r f (\pi)\ran+ \sum_{\pi \sim_{q-r} \pi_0} \lan f \star_r^r f, f \star_r^r f (\pi)\ran\right.\\
&\qquad\qquad\qquad\qquad\qquad\qquad \left.+ \sum_{p=1}^{q-r-1} \sum_{\pi\sim_p \pi_0} \lan f \star_r^r f, f \star_r^r f (\pi)\ran\right)  \\
&\le  \frac{1}{(2q-2r)!} \left( 2 ((q-r)!)^2 \|f\star_r^rf\|^2 +  \sum_{p=1}^{q-r-1} ((q-r)!)^2 \binom{q-r}{p}^2 \|f\star_p^p f\|^2 \right).
\end{align*} 
This completes the proof.
\end{proof}

\begin{remark}\label{remark:sigma}
A combinatorial argument shows that the permutation $\sigma\in \Pi_{2q-2p}$ appearing in \eqref{eqn:computation} cannot be sucht that $f \star_p^p f(\sigma) = f \star_p^p f$ (in particular, $\sigma$ cannot be the identity). Hence, we cannot omit applying Cauchy-Schwarz in this case.
\end{remark}

In Lemma \ref{lemma:symmetrization} no condition on the sign of $f$ was necessary. However, if we assume that $f$ has constant sign, we are able to deduce a `reverse' counterpart of \eqref{eqn:r-contraction}.

\begin{lemma}\label{lemma:reverse inequality}
Let $q\ge1$ be an integer and $f\in L_s^2(\mu^q)$ a kernel satisfying the technical assumptions (A). If $f\le0$ or $f\ge0$, then, for any $r\in\{0,1,\ldots, q-1\}$, one has that
\begin{align}\label{eqn:reverse inequality}
\|f \,\tld\star_r^r f\|^2 \ge \frac{2((q-r)!)^2}{(2q-2r)!} \|f \star_r^r f\|^2\,.
\end{align}
\end{lemma}

\begin{proof}
The left-hand side of \eqref{eqn:reverse inequality} satisfies identity \eqref{eqn:norm-expansion}. Using the fact that $f$ has constant sign, the right-hand side of \eqref{eqn:norm-expansion} becomes smaller if we sum only over a subset of $\Pi_{2q-2r}$, namely over all $\pi\in\Pi_{2q-2r}$ such that $\pi \sim_0 \pi_0$ or $\pi\sim_{q-r} \pi_0$. Hence, we end up with
\begin{align*}
\|f\,\tld \star_r^r f \|^2 
&\ge\frac{1}{(2q-2r)!} \left( \sum_{\pi \sim_0 \pi_0} \lan f \star_r^r f, f \star_r^r f (\pi)\ran + \sum_{\pi \sim_{q-r} \pi_0} \lan f \star_r^r f, f \star_r^r f (\pi)\ran \right) \\
&= \frac{2((q-r)!)^2}{(2(q-r))!} \|f \star_r^r f\|^2\,,
\end{align*}
which completes the proof.
\end{proof}

\begin{remark}\label{remark:reverse inequality}
In view of Remark \ref{remark:sigma}, inequality \eqref{eqn:reverse inequality} is optimal under the conditions of Lemma \ref{lemma:reverse inequality}.
\end{remark}

\subsection{Proof of the implication (ii) $\Longrightarrow$ (iii)}

Let us introduce some notation. We shall write $a_n \asymp b_n$ for two real-valued sequences $\{a_n\colon n\ge1\}$, $\{b_n\colon n\ge1\}$, whenever $\lim\limits_{n\to\infty} a_n - b_n=0$. Be aware that this does not necessarily imply that one of the individual sequences converges, but of course ensures the convergence of both sequences whenever one of them converges.

The next lemma establishes an asymptotic lower bound for the linear combination of the fourth and third moment $\EE[I_q(f_n)^4] - 12 \EE[I_q(f_n)^3]$ of a sequence of Poisson integrals of even order $q\ge2$ where $\{f_n \colon n\ge1\}\subset L_s^2(\mu_n^q)$. It is one of the main ingredients to show the implication (ii) $\Longrightarrow$ (iii) in Theorem \ref{thm:main_theorem}. Note that this bound holds for general even $q\geq 2$. Moreover, at this point we do not need an assumption on the sign of the kernels.

\begin{lemma}\label{lemma:A+R}
Let $\nu>0$ and $q\ge2$ be an even integer. Let $\{f_n\colon n\ge1\} \subset L_s^2(\mu_n^q)$ be a sequence of kernels such that the technical assumptions (A) and the normalization condition
\[
\lim_{n\to\infty} q! \|f_n\|^2 = 2\nu
\]
are satisfied.
Then one has that
\be{eqn:MomentComputation}
\EE[I_q(f_n)^4] - 12 \EE[I_q(f_n)^3]\asymp 12\nu^2 - 48 \nu + A(I_q(f_n)) + R(I_q(f_n))\,,
\ee
where the terms on the right-hand side of \eqref{eqn:MomentComputation} satisfy $A(I_q(f_n))\ge A'(I_q(f_n))$ with
\be{eqn:A'}
\begin{split}
A'(I_q(f_n))&=\sum_{p=1}^{q/2 - 1} \frac{(q!)^4}{(p!)^2}\left(\frac{2}{(q-p)!^2} - \frac{1}{2\big((q/2)!(q/2 - p)!\big)^2} \right) \|f_n\star_p^p f_n \|^2\\
&\qquad+ \sum_{p=1, p\neq q}^{2q-1} p!  \|G_p^q\,f_n \|^2+ q! \sum_{p=q/2+1}^q (p!)^2 \binom{p}{q}^4 \binom{p}{q-p}^2 \|f_n\, \tld\star_p^{q-p}\,f_n \|^2\\
& \qquad+ 24q! \|c_q^{-1} f_n\, \tld\star_{q/2}^{q/2}\,f_n - f_n\|^2
\end{split}
\ee
with $c_q = \frac{4}{(q/2)! \binom{q}{q/2}^2}$, and
\be{eqn:R}
\begin{split}
R(I_q(f_n))&=q! \sum_{\substack{r,p=q/2 \\ r\neq p}}^q r!\, p! \binom{q}{r}^2 \binom{q}{p}^2 \binom{r}{q-r} \binom{p}{q-p}
 \lan f_n \,\tld \star_{r}^{q - r} f_n , f_n \,\tld \star_{p}^{q - p} f_n\, \ran \\
 &\quad-12 q! \sum_{p=q/2+1}^q p! \binom{q}{p}^2 \binom{p}{q-p} 
 \lan f_n \,\tld \star_{p}^{q - p} f_n,f_n\, \ran \,.
\end{split}
\ee
\end{lemma}

\begin{proof}[Proof of Lemma \ref{lemma:A+R}]
In view of Lemma \ref{lemma:Third and fourth moment} and since $q$ is even, one has that
\begin{align*}
\begin{split}
&\EE[I_q(f_n)^4] - 12 \EE[I_q(f_n)^3] \\
&= \sum_{p=0}^{2q} p! \|G_p^q\, f_n \|^2 - 
12 q! \sum_{p=q/2}^q p! \binom qp^2 \binom{p}{q-p}  \lan f_n \,\tld \star_{p}^{q-p} f_n,f_n\, \ran \\
%cite[equation 11.6.30]{PeccatiTaqquBook}
&= (q!)^2 \|f_n \|^4 + (2q)! \|f_n \,\tld\star_0^0 f_n\|^2	+ \sum_{p=1}^{2q-1} p! \|G_p^q f_n\|^2 \\
	&\quad -12 q! \sum_{p=q/2}^q p! \binom qp^2 \binom{p}{q-p}  \lan f_n \,\tld \star_{p}^{q-p} f_n,f_n\, \ran \\
&= 3(q!)^2 \|f_n \|^4 + \sum_{p=1}^{q-1} \frac{(q!)^4}{\big(p! (q - p)! \big)^2} \| f_n \star_p^p f_n \|^2
	+ \sum_{p=1}^{2q-1} p! \|G_p^q f_n\|^2 \\
	&\quad -12 q! \sum_{p=q/2}^q p! \binom qp^2 \binom{p}{q-p}  \lan f_n \,\tld \star_{p}^{q-p} f_n,f_n\, \ran \\
&= 3(q!)^2 \|f_n \|^4 + T_1(I_q(f_n)) + T_2(I_q(f_n)) + T_3(I_q(f_n))\,,
\end{split}
\end{align*}
where the third equality stems from \eqref{eqn:0-contraction}. The terms $T_1, T_2, T_3$ read as follows:
\begin{align*}
T_1(I_q(f_n))&= \sum_{\substack{p=1 \\ p\neq q/2}}^{q-1}   \frac{(q!)^4}{\big(p! (q - p)! \big)^2} \| f_n \star_p^p f_n \|^2 
	+ \sum_{p=1, p\neq q}^{2q-1} p! \|G_p^q f_n\|^2, \\
T_2(I_q(f_n))&= \frac{(q!)^4}{(q/2)!^4} \| f_n \star_{q/2}^{q/2} f_n \|^2 + q! \|G_q^q f_n\|^2 
	- 12 q! (q/2)! \binom{q}{q/2}^2 \lan f_n \,\tld \star_{q/2}^{q/2} f_n,f_n\, \ran , \\
T_3(I_q(f_n))&=  -12 q! \sum_{p=q/2+1}^q p! \binom qp^2 \binom{p}{q-p}  \lan f_n \,\tld \star_{p}^{q-p} f_n,f_n\, \ran .
\end{align*}
%Now, we use the fact that we can exactly compare the norm of the `middle-contraction' and its symmetrized version. We have that $\| f_n \star_{q/2}^{q/2} f_n \|^2 = \frac12 \binom{q}{q/2} \| f_n\, \tld \star_{q/2}^{q/2} f_n \|^2$.
%\TF{Hast du eine Referenz fuer diesen Zusammenhang? ACHTUNG: Dieser Zusammenhang gilt erst einmal nur unter der Annahme, dass $\|f_n \star_p^p f_n\|\to0$ for all $p\in\{1,\ldots, q-1\}\setminus \{q/2\}$!!!}

We use \eqref{eqn:q/2-contraction} to see that
\[
 \frac{(q!)^4}{(q/2)!^4} \| f_n \star_{q/2}^{q/2} f_n \|^2 \ge 
 \frac{(q!)^5}{2(q/2)!^6} \|f_n\, \tld \star_{q/2}^{q/2} f_n\|^2 - \frac12 \sum_{p=1}^{q/2-1} \frac{(q!)^4}{\big((q/2)! p!(q/2-p)!\big)^2}\|f_n\star_p^p f_n \|^2.
\]

Using the definition of $G_q^q f_n$ given at \eqref{eqn:G_p^q}, we have the estimate
\begin{align*}
T_2(I_q(f_n))&\ge q! \left(
	\big \| \sum_{r=q/2}^q r! \binom{q}{r}^2 \binom{r}{q-r} f_n \, \tld\star_r^{q-r} f_n \big\|^2 
	+ \frac12 \frac{(q!)^4}{(q/2)!^6} \|f_n \, \tld\star_{q/2}^{q/2} f_n\|^2 \right.\\
	&\quad \left. - 12  (q/2)! \binom{q}{q/2}^2 \lan f_n \,\tld \star_{q/2}^{q/2} f_n,f_n\, \ran
	\right) 
	- \frac12 \sum_{p=1}^{q/2-1} \frac{(q!)^4}{\big((q/2)! p!(q/2-p)!\big)^2}\|f_n\star_p^p f_n \|^2 \\
&= q! \left(
	\frac32 \frac{(q!)^4}{(q/2)!^6} \|f_n \, \tld\star_{q/2}^{q/2} f_n\|^2 
	- 12   \frac{(q!)^2}{(q/2)!^3} \lan f_n \,\tld \star_{q/2}^{q/2} f_n,f_n\, \ran
	\right) \\
	&\quad + q!\sum_{r=q/2+1}^q (r!)^2 \binom{q}{r}^4 \binom{r}{q-r}^2 \| f_n \, \tld\star_r^{q-r} f_n\|^2 \\
	&\quad + q!\sum_{\substack{r,p=q/2 \\ r\neq p}}^q r!p!  \binom{q}{r}^2 \binom{q}{p}^2 \binom{r}{q-r} \binom{p}{q-p}
	 \lan f_n \,\tld \star_{r}^{q - r} f_n , f_n \,\tld \star_{p}^{q - p} f_n\, \ran \\
	 &\quad- \frac12 \sum_{p=1}^{q/2-1} \frac{(q!)^4}{\big((q/2)! p!(q/2-p)!\big)^2}\|f_n\star_p^p f_n \|^2.
\end{align*}
Using the relation $\| f_n\star_p^p f_n\| = \|f_n \star_{q-p}^{q-p} f_n\|$, valid for all $p\in\{1,\ldots,q-1\}$, we obtain
\begin{align*}
&\sum_{\substack{p=1 \\ p\neq q/2}}^{q-1}   \frac{(q!)^4}{\big(p! (q - p)! \big)^2} \| f_n \star_p^p f_n \|^2 
	-  \frac12 \sum_{p=1}^{q/2-1} \frac{(q!)^4}{\big((q/2)! p!(q/2-p)!\big)^2}\|f_n\star_p^p f_n \|^2 \\
=&  \sum_{p=1}^{q/2 - 1} \frac{(q!)^4}{(p!)^2}\left(\frac{2}{(q-p)!^2} - \frac{1}{2\big((q/2)!(q/2 - p)!\big)^2} \right) \|f_n\star_p^p f_n \|^2.
\end{align*}
The proof is concluded by observing that
\begin{align*}
&q! \left(
\frac32 \frac{(q!)^4}{((q/2)!)^6} \|f_n \, \tld\star_{q/2}^{q/2} f_n\|^2 
	- 12   \frac{(q!)^2}{((q/2)!)^3} \lan f_n \,\tld \star_{q/2}^{q/2} f_n,f_n\, \ran
	\right) \\
&= \frac32 q! \left(  \frac{(q!)^4}{((q/2)!)^6} \|f_n \, \tld\star_{q/2}^{q/2} f_n\|^2 
	- 2\times 4 \frac{(q!)^2}{((q/2)!)^3} \lan f_n \,\tld \star_{q/2}^{q/2} f_n,f_n\, \ran
	+ 16 \|f_n\|^2
	\right) 
	-24q! \|f_n \|^2 \\
&=24q! \|c_q^{-1} f_n\, \tld\star_{q/2}^{q/2}\,f_n - f_n\|^2 - 24q! \|f_n \|^2
\end{align*}
and by recalling condition (a), which implies that $3(q!)^2 \|f_n \|^4 - 24q! \|f_n \|^2 \to 12\nu^2 - 48\nu$.
\end{proof}

While all previous results did not use the assumptions on the order of the integral and the sign of the kernels, in the next lemma we need that $q\in\{2,4\}$ and that the kernels have constant sign.

\begin{lemma}\label{lemma:A>=0}
Let $\nu>0$ and $q\in\{2,4\}$. Let $\{f_n\colon n\ge1\} \subset L_s^2(\mu_n^q)$ be a sequence of kernels such that the technical assumptions (A) and the normalization condition $\lim\limits_{n\to\infty} q! \|f_n\|^2 = 2\nu$ are satisfied. Assume that for each $n\ge1$ either $f_n\le 0$ or $f_n\ge0$. Then the following two assertions concerning the term $A'(I_q(f_n))$ defined at \eqref{eqn:A'} are true:
\begin{enumerate}
\item[(1)]
$A'(I_q(f_n))\ge0$ for all $n\geq 1$;
\item[(2)]
If $A'(I_q(f_n))\to0$, as $n\to\infty$, then 
\begin{equation}\label{eqn:contraction_condition}
\lim_{n\to\infty} \|f_n \star_r^\ell f_n\| = 0
\end{equation}
for all $r\in\{1, \ldots, q\}$ and $\ell\in\{1, \ldots, r\wedge(q-1)\}$ such that $(r,\ell)\neq(q/2,q/2)$, 
\begin{align}
& \lim_{n\to\infty} \|f_n^2\| =0, \label{eqn:L^4}\\
&\lim_{n\to\infty} \|f_n\, \tld\star_{q/2}^{q/2}\,f_n - c_q f_n\|=0 \quad\text{ with } c_q = \frac{4}{(q/2)! \binom{q}{q/2}^2}\,.
\label{eqn:middle_contraction}
\end{align}
\end{enumerate}
\end{lemma}

\begin{proof}%[Proof of Lemma \ref{lemma:A>=0} ]
We start by showing the first assertion of the lemma. The only term that might be negative on right-hand side of \eqref{eqn:A'} is the first sum. For the case $q=2$, this does not play any role, because then the sum vanishes. Hence,  $A'(I_q(f_n))$ is a positive linear combination of non-negative terms.

Now, let $q\ge4$ be even. Using the fact that $f_n$ has constant sign, $\|f_n\star_{p}^{p} f_n\| = \|f_n\star_{q-p}^{q-p} f_n\|$ for all $p\in \{1,\ldots, q-1\}$ as well as Lemma \ref{lemma:reverse inequality}, we obtain the estimate
\begin{align*}
 \sum_{p=1,\, p\neq q}^{2q-1} p!  \|G_p^q\,f_n \|^2
 \ge& \sum_{p=1,\, p\neq q}^{2q-1} p! \sum_{r=0}^q \sum_{\ell=0}^r \one(2q - r - \ell=p) r!^2 \binom qr^4 \binom r\ell^2 \|f_n \,\tld\star_r^\ell f_n\|^2  \\
 \ge& \sum_{\substack{p=1, \,p\neq q, \\p \text{ even}}}^{2q-1} p! ((q-p/2)!)^2 \binom{q}{q-p/2}^4 \|f_n \,\tld\star_{q-p/2}^{q-p/2} f_n\|^2  \\
  =& \sum_{p=1,\, p\neq q/2}^{q-1} (2p)! ((q-p)!)^2 \binom{q}{q-p}^4 \|f_n \,\tld\star_{q-p}^{q-p} f_n\|^2  \\
  =& \sum_{p=1, \,p\neq q/2}^{q-1} (2(q-p))! (p!)^2 \binom{q}{p}^4 \|f_n \,\tld\star_{p}^{p} f_n\|^2  \\
  \ge& \sum_{p=1,\, p\neq q/2}^{q-1} 2((q-p)!)^2 (p!)^2 \binom{q}{p}^4 \|f_n \star_{p}^{p} f_n\|^2  \\
  =& \sum_{p=1}^{q/2-1} \frac{4(q!)^4}{((q-p)!)^2 (p!)^2} \|f_n \star_{p}^{p} f_n\|^2.
\end{align*}
Hence, we end up with
\begin{align}\nonumber
&\sum_{p=1}^{q/2 - 1} \frac{(q!)^4}{(p!)^2}\left(\frac{2}{(q-p)!^2} - \frac{1}{2\big((q/2)!(q/2 - p)!\big)^2} \right) \|f_n\star_p^p f_n \|^2 
	+ \sum_{p=1, p\neq q}^{2q-1} p!  \|G_p^q\,f_n \|^2 \\
\ge&\sum_{p=1}^{q/2 - 1} \frac{(q!)^4}{(p!)^2}\left(\frac{6}{(q-p)!^2} - \frac{1}{2\big((q/2)!(q/2 - p)!\big)^2} \right) \|f_n\star_p^p f_n \|^2.\label{eqn:coefficient}
\end{align}
For $q=4, p=1$ we have that 
\[
\frac{6}{(q-p)!^2} - \frac{1}{2\big((q/2)!(q/2 - p)!\big)^2}=\frac{1}{24}>0\,.
\]
So, for $q=4$ (and $q=2$), the term $A'(I_q(f_n))$ is bounded from below by a linear combination with positive coefficients of the norms of the contraction kernels appearing in \eqref{eqn:contraction_condition}, \eqref{eqn:L^4} and \eqref{eqn:middle_contraction} (while for all even $q\geq 6$ this cannot be guaranteed any more). This proves both statements of the lemma.
\end{proof}

\begin{remark}
As anticipated, for all even $q\geq 6$ there are combinatorial coefficients in \eqref{eqn:coefficient} which are negative, implying that our proof cannot be generalized to Poisson integrals of arbitrary order. The reason is that one would need a sharper version of Lemma \ref{lemma:reverse inequality}, which is in general not available as discussed in Remark \ref{remark:reverse inequality}. As a consequence, we have to leave it as an open problem to establish a four moments theorem for the Gamma approximation for Poisson integrals of order $q\geq 6$ by different methods.
\end{remark}

It remains to check whether the conditions of Theorem \ref{thm:main_theorem} are sufficient to imply that $A'(I_q(f_n))\to0$. The following lemma shows that this is indeed the case.

\begin{lemma}\label{lemma:A to 0}
Let $\nu>0$ and $q\in\{2,4\}$. Let $\{f_n\colon n\ge1\} \subset L_s^2(\mu_n^q)$ be a sequence of kernels satisfying the technical assumptions (A) and the normalization condition
\[
\lim_{n\to\infty} q!\|f_n\|^2 = \lim_{n\to\infty}\EE[I_q(f_n)^2] = 2\nu\,.
\]
Let the sequence $\{I_q(f_n)^4\colon n\ge1\}$ be uniformly integrable. If one of the conditions
\begin{enumerate}
\item
$q=2$ and $\lim\limits_{n\to\infty} \|f_n^2\|=0$,
\item
$q=4$ and $f_n\le 0$ for all $n\ge1$,
\end{enumerate}
is satisfied, then the following implication is true. If
\[
\lim_{n\to\infty} \EE[I_q(f_n)^4] - 12 \EE[I_q(f_n)^3] =  12\nu^2 - 48 \nu 
\]
then
\begin{equation}\label{eqn:contraction_condition b}
\lim_{n\to\infty} \|f_n \star_r^\ell f_n\| = 0
\end{equation}
for all $r\in\{1, \ldots, q\}$ and $\ell\in\{1, \ldots, r\wedge(q-1)\}$ such that $(r,\ell)\neq(q/2,q/2)$, 
\begin{align}
& \lim_{n\to\infty} \|f_n^2\| =0, \label{eqn:L^4 b}\\
&\lim_{n\to\infty} \|f_n\, \tld\star_{q/2}^{q/2}\,f_n - c_q f_n\|=0 \quad\text{ with } c_q = \frac{4}{(q/2)! \binom{q}{q/2}^2}\,.
\label{eqn:middle_contraction b}
\end{align}
\end{lemma}

\begin{proof}
First apply Lemma \ref{lemma:A+R} to deduce that $ A(I_q(f_n)) + R(I_q(f_n)) \to0$, as $n\to\infty$.

Assume that $q=2$ and $\|f_n^2\|\to0$. Then \eqref{eqn:L^4 b} is satisfied by assumption. Moreover,
\begin{align*}
R(I_2(f_n)) &= 32 \lan f_n \,\tld \star_1^1 f_n, f_n\,\tld\star_2^0 f_n\ran
	-48  \lan f_n \,\tld \star_2^0 f_n,  f_n\ran.
\end{align*}
By the Cauchy-Schwarz inequality, we see that
\[
| \lan f_n \,\tld \star_1^1 f_n, f_n\,\tld\star_2^0 f_n\ran | \le \| f_n \,\tld \star_1^1 f_n\| \, \|f_n\,\tld\star_2^0 f_n \|, 
\qquad
|\lan f_n \,\tld \star_2^0 f_n,  f_n\ran| \le \|f_n\|\,\|f_n\,\tld\star_2^0 f_n\|.
\]
With respect to the definition of the contractions, we see that $f_n \,\tld \star_2^0 f_n = f_n^2$. We shall argue now that the sequence $\| f_n \,\tld \star_1^1 f_n\| $ is bounded. For this, observe that for any fixed $(s,t) \in \cZ^2$, we obtain by the Cauchy-Schwarz inequality that 
\begin{align*}
|f_n\star_1^1 f_n (t,s) | &= \Big| \int_\cZ f_n (z,t) f_n(z,s) \mu(\dint z) \Big| \\
&\le \Big(  \int_\cZ f_n^2 (z,t) \mu(\dint z) \Big)^{1/2} \Big(  \int_\cZ f_n^2 (z,s) \mu(\dint z) \Big)^{1/2}.
\end{align*}
Consequently, 
\begin{align*}
\| f_n \,\tld \star_1^1 f_n\|^2 \le  \| f_n \star_1^1 f_n\|^2 
= \int_{\cZ^2} |f_n\star_1^1 f_n (t,s) |^2 \mu^2(\dint (s,t) )
\le \|f_n\|^4.
\end{align*}
By assumption, we have that $\|f_n\|^2\to \nu$, so the sequence is bounded. Now, the fact that $\|f_n^2\|\to0$ implies that $R(I_2(f_n))\to0$. Hence, $A(I_2(f_n))\to0$, which implies that $A'(I_2(f_n))\to0$ using Lemma \ref{lemma:A>=0}(1). Now, we apply Lemma \ref{lemma:A>=0}(2) to see that \eqref{eqn:contraction_condition b} and \eqref{eqn:middle_contraction b} follow.

Next, let $q=4$ and suppose that $f_n\le0$. Recall that the tensor product is bi-linear and it is easily verified that the contraction operation preserves this bi-linearity. Now, the fact that the kernels are non-positive ensures that $R(I_q(f_n))\ge0$ and we can again apply Lemma \ref{lemma:A>=0}(1) to see that $0\le A'(I_q(f_n))\le A(I_q(f_n))$. Hence, we deduce that $A(I_q(f_n))\to0$. This directly implies that $A'(I_q(f_n))\to0$, such that the claim follows again by Lemma \ref{lemma:A>=0}(2). 
\end{proof}

\begin{remark}\label{remark:symmetry}
Let us explain in some more detail why in contrast to the case of normal approximation the kernels have to be non-positive for Gamma approximations. An inspection of the proof of Theorem \ref{thm:main_theorem} shows that a constant sign of the kernels is necessary to control the sign of scalar products. This is necessary in Lemma \ref{lemma:reverse inequality} and therefore also in Lemma \ref{lemma:A>=0} to control the signs of $A(I_q(f_n))$ and $A'(I_q(f_n))$, respectively. On the other hand, this is also necessary in part b) of Lemma \ref{lemma:A to 0}, where one has to control the sign of $R(I_q(f_n))$. In this context, scalar products of the form $\lan f_n \,\tld\star_p^{q-p} f_n,f_n\ran$, $p\in\{q/2+1, \ldots, q\}$, appear. They are thrice-linear in $f_n$, such that $f_n\le0$ implies that $\lan f_n \,\tld\star_p^{q-p} f_n,f_n\ran\le0$ and we can conclude that $R(I_q(f_n))\ge0$. In summary, knowing that $A'(I_q(f_n))\ge0$ and $R(I_q(f_n))\ge0$ enables us to use part (2) of Lemma \ref{lemma:A>=0} to get the implication (ii) $\Longrightarrow$ (iii) in Theorem \ref{thm:main_theorem}. Note that the latter scalar products in $R(I_q(f_n))$ actually stem from the third moment in assertion (ii) of Theorem \ref{thm:main_theorem} (see also \eqref{eqn:3moment even}).

It is worth mentioning that this asymmetry in the assertions for Theorem \ref{thm:main_theorem} (and also in Theorem \ref{thm:Gaussian-Gamma}) is actually an intrinsic property of the Gamma distribution which contrasts the normal case. For the central limit theorem in a Poisson chaos, it can be easily seen that if the law of the sequence $\{I_q(f_n)\colon n\ge1\}$ converges to a standard normal law $\cN(0,1)$, then also the law $\{I_q(-f_n)\colon n\ge1\}=\{-I_q(f_n)\colon n\ge1\}$ converges to $\cN(0,1)$, since the standard normal law is symmetric. Consistently, assertions (ii) and (iii) in the four moments theorem for normal approximation are invariant under sign changes of the kernels. In sharp contrast, the assertions for Gamma approximations are not invariant under such sign changes, because of the lack of symmetry of the target distribution. This means that if the law of $\{I_q(f_n)\colon n\ge1\}$ converges to $\overline \G_\nu$ then that law of $\{I_q(-f_n)\colon n\ge1\}=\{-I_q(f_n)\colon n\ge1\}$ cannot converge to $\overline \G_\nu$. Consistently, assertions (ii) and (iii) in Theorem \ref{thm:main_theorem} inherit this asymmetry, which is reflected by the appearance of the third moment in (ii) and the term $\|f_n \,\tld\star_{q/2}^{q/2} f_n - c_q \,f_n\|$ in (iii), both of them not being invariant under a change of the sign of $f_n$.
\end{remark}

\subsection{An alternative approach to the four moments theorem}
 
In Remark \ref{remark:symmetry} we explained that the sign condition on the kernels in part (b) of Theorem \ref{thm:main_theorem} ensures that $R(I_q(f_n))\ge0$. Together with $A'(I_q(f_n))\ge0$, this is sufficient in combination with part (2) of Lemma \ref{lemma:A>=0} to get the implication (ii) $\Longrightarrow$ (iii) in Theorem \ref{thm:main_theorem}. On the other hand, for part (a) of Theorem \ref{thm:main_theorem}, dealing with the case $q=2$, the assumption that $\|f_n^2\|\to0$ yields that $R(I_2(f_n))\to0$, an assertion also being sufficient in combination with $A'(I_q(f_n))\ge0$ to deduce the implication (ii) $\Longrightarrow$ (iii) in Theorem \ref{thm:main_theorem} from part (2) of Lemma \ref{lemma:A>=0}. From this point of view, it is natural to ask whether the latter condition can be generalized to arbitrary $q\geq 2$. Our next result shows that this is indeed possible, but leads to a result which is weaker than Theorem \ref{thm:main_theorem}. Moreover, the proof again only works for $q=4$ and we still have to impose a sign condition on the sequence of kernels.

\begin{proposition}\label{prop:norms}
Fix $\nu>0$. Let $\{f_n\colon n\ge1\} \subset L_s^2(\mu_n^4)$ be a sequence of kernels such that $f_n\ge0$ for all $n\ge1$ and such that the technical assumptions (A) and the normalization condition
\[
\lim_{n\to\infty} 4!\|f_n\|^2 = \lim_{n\to\infty}\EE[I_4(f_n)^2] = 2\nu
\]
are satisfied. Assume additionally that
\be{eqn:norms}
\lim_{n\to\infty} \|f_n^2\|=0, \text{ and }
\lim_{n\to\infty} \|f_n \star_3^{1} f_n\|=0.
\ee
If the sequence $\{I_4(f_n)^4\colon n\ge1\}$ is uniformly integrable, then the equivalence stated in Theorem \ref{thm:main_theorem} remains valid.
\end{proposition}

\begin{proof}
The implication (i) $\Longrightarrow$ (ii) follows from the uniform integrability of the sequence $\{I_4(f_n)^4\colon n\ge1\}$ and (iii) $\Longrightarrow$ (i) is a consequence of Proposition \ref{prop: Theorem 2.6}. To establish the implication (ii) $\Longrightarrow$ (iii), we apply Lemma \ref{lemma:A+R} and show that the term $R(I_4(f_n))$ defined at \eqref{eqn:R} converges to zero, as $n\to\infty$.
 With the Cauchy-Schwarz inequality we obtain for $p\in\{3,4\}$ that 
\[
|\lan f_n\,\tld \star_p^{4-p} f_n, f_n\ran| \le \|f_n\star_p^{4-p} f_n\|\,\|f_n\|\to0\,,
\]
since $\|f_n\star_p^{4-p} f_n\|\to0$ and $\|f_n\|^2\to\frac{\nu}{12}$. Moreover, for $p,r\in\{2,3,4\}$ with $p\neq r$ we also get
\[
|\lan f_n\,\tld \star_p^{4-p} f_n, f_n \,\tld\star_r^{4-r}f_n \ran|\le \| f_n\,\tld \star_p^{4-p} f_n \| \, \| f_n \,\tld\star_r^{4-r}f_n \|\to0\,.
\]
The convergence is ensured by condition \eqref{eqn:norms} if $p,r>2$. If otherwise $p\wedge r=2$, we use condition \eqref{eqn:norms} together with the observation that $\|f_n\star_2^0f_n\|=\|f_n\star_4^2f_n\|$ and $\|f_n\star_2^2f_n\|\leq\|f_n\star_4^4f_n\|$ as a consequence of Fubini's theorem and the Cauchy-Schwarz inequality. Summarizing, we see that $R(I_4(f_n))\to0$, which in turn implies that $A(I_q(f_n))\to0$ thanks to Lemma \ref{lemma:A+R}. We can then conclude as in the proof of part (b) of Lemma \ref{lemma:A to 0}.
\end{proof}

\begin{acknowledgements}
We would like to thank Johanna F.\ Ziegel for initiating this collaboration and Giovanni Peccati for helpful discussions and valuable comments.
\end{acknowledgements}

%\bibliography{Gamma3}

\end{document}